\documentclass[10pt]{article}
\usepackage[T1]{fontenc}
\usepackage[latin1]{inputenc}
\usepackage{amssymb,amsmath,amsthm,eucal,mathrsfs,array,graphicx}
\usepackage[all]{xy}
\usepackage{bbm}

\newtheorem{thm}{Theorem}[section]
\newtheorem{lem}[thm]{Lemma}
\newtheorem{rem}[thm]{Remark}
\newtheorem{prop}[thm]{Proposition}

\newtheorem*{thm0}{Theorem \ref{thm:KSloop-2}}
\newtheorem*{thm1}{Theorem \ref{thm:cv-loop}}
\newtheorem*{thm2}{Theorem \ref{thm:cv-expl}}
\newtheorem*{thm3}{Lemma \ref{lem:cv-explafterT}}
\newtheorem*{thm4}{Lemma \ref{lem:charact-cv-expl}}
\newtheorem*{thm5}{Lemma \ref{lem:DRN-expl}}
\newtheorem*{thm6}{Theorem \ref{thm:cv-explbubble}}
\newtheorem*{thm7}{Proposition \ref{prop:cv-det}}

\theoremstyle{definition}

\newtheorem{defn}[thm]{Definition}

\def \a {\alpha}

\def \g {\gamma}
\def \d {\delta}
\def \D {\Delta}
\def \e {\varepsilon}

\def \s {\sigma}
\def \Si {\Sigma}

\def \O {\Omega}

\def \z {\zeta}      
\def \Co {\mathbb{C}}
\def \Di {\mathbb{D}} 

\def \H {\mathbb{H}}
 
\def \N {\mathbb{N}} 
\def \P {\mathbb{P}} 
 
\def \R {\mathbb{R}} 
\def \S {\mathbb{S}}
\def \Z {\mathbb{Z}}
\def \ind {\mathbbm{1}}
\def \A {\mathcal{A}}
\def \B {\mathcal{B}}
\def \C {\mathcal{C}} 
\def \F {\mathcal{F}}
\def \G {\mathcal{G}}
\def \Ha {\mathcal{H}}
\def \L {\mathcal{L}}
\def \Me {\mathcal{M}}

\def \Top {\mathcal{T}}

\def \Id {{\rm Id}}
\def \l {\ell}
\def \lo {\mathrm{loop}}

\def \M {\mathrm{M}}

\def \SLE {\mathrm{SLE}}
\def \T {\mathrm{T}}
\def \Tr {\mathrm{Tr}}
\def \sm {\setminus}
\def \arcperso {\overline}

\newcommand{\mc}{\mathcal}

\DeclareMathOperator{\Det}{Det}

\title{An SLE$_2$ loop measure}
\author{St\'ephane Benoist\\Julien Dub\'edat\footnote{Partially supported by NSF grant DMS-1005749.}}

\begin{document}

\maketitle

\begin{abstract}
There is an essentially unique way to associate to any Riemann surface a measure on its simple loops, such that the collection of measures satisfy a strong conformal invariance property (see \cite{Wer_loops}). These random loops are constructed as the boundary of Brownian loops, and so correspond in the zoo of statistical mechanics models to central charge $0$, or Schramm-Loewner Evolution (SLE) parameter $\kappa=8/3$.
The goal of this paper is to construct a family of measures on simple loops on Riemann surfaces that satisfies a conformal covariance property, and that would correspond to SLE parameter $\kappa=2$ (central charge $-2$).  On planar annuli, this loop measure was already built by Adrien Kassel and Rick Kenyon in \cite{KasKen_RandomCurves}. We will give an alternative construction of this loop measure on planar annuli, investigate its conformal covariance, and finally extend this measure to general Riemann surfaces. This gives an example of a Malliavin-Kontsevich-Suhov loop measure \cite{KontSuh} in non-zero central charge.
\end{abstract}

\tableofcontents

\section{Introduction}
Our goal is to construct a family of measures on simple loops on Riemann surfaces related to $\SLE_2$. To each Riemann surface $\Sigma$, one associates a measure on the space $\L(\Sigma)$ of its (non-oriented) simple loops (i.e. the space of injective maps $\S^1 \rightarrow \Sigma$, up to increasing or decreasing reparametrization), satisfying a certain, central-charge dependent restriction condition when comparing the measure on a surface $\Sigma$ to the one on $\Sigma'$, whenever $\Sigma'\subset\Sigma$. See Section \ref{from-ann-to-rs} for a detailed discussion.

\begin{thm}\label{thm:KSloop-2}
There exists a $c$-locally conformally covariant loop measure, in the sense of Kontsevich and Suhov (Definition \ref{def:c-lcc}), with parameter $c=-2$.
\end{thm}

The parameter $c$ can be interpreted as the central charge of field theory (see for example \cite{Dub_SLEVirloc}). At $c=0$, existence and uniqueness was established earlier by Werner \cite{Wer_loops}. The present result yields existence at $c=-2$; existence and uniqueness are conjectured to hold for $c\leq 1$ in \cite{KontSuh}. Via welding, a (finite) measure on simple loops induces a measure on homeomorphisms of the unit circle, a problem initially considered by Malliavin (\cite{Mal_diffcircle}). Here, the measures are supported on loops that are, in a loose sense, locally absolutely continuous with respect to $\SLE_2$. 

We are first going to construct these measures on topologically non-trivial loops drawn in piecewise-$\C^1$ conformal annuli, as limits of random loops on discrete graphs. Consider a conformal planar annulus $A$ with piecewise $\C^1$ boundary, and let us consider the natural approximation of $A$ by a family $A^\d$ of finite subgraphs of $\d\Z^2$ (Definition \ref{natural-approx}). On such a discrete annuli $A^\d$, consider the wired uniform spanning tree (or UST, see Definition \ref{def:UST}), which is a random subgraph of $A^\d$.

The wired UST on $A^\d$ has two connected components, one attached to the outer boundary of the annulus, the other attached to its inner boundary. These two connected components are in contact with each other alongside a simple closed curve $\l^\d_{A^\d}$, that winds once around the central hole. The loop $\l^\d_{A^\d}$ can equivalently be seen as the unique cycle in the subgraph of $\d\Z^2 + (\d/2,\d/2)$ dual to the spanning tree.
\begin{figure}[htb]
\begin{center}
\includegraphics[width = 10cm]{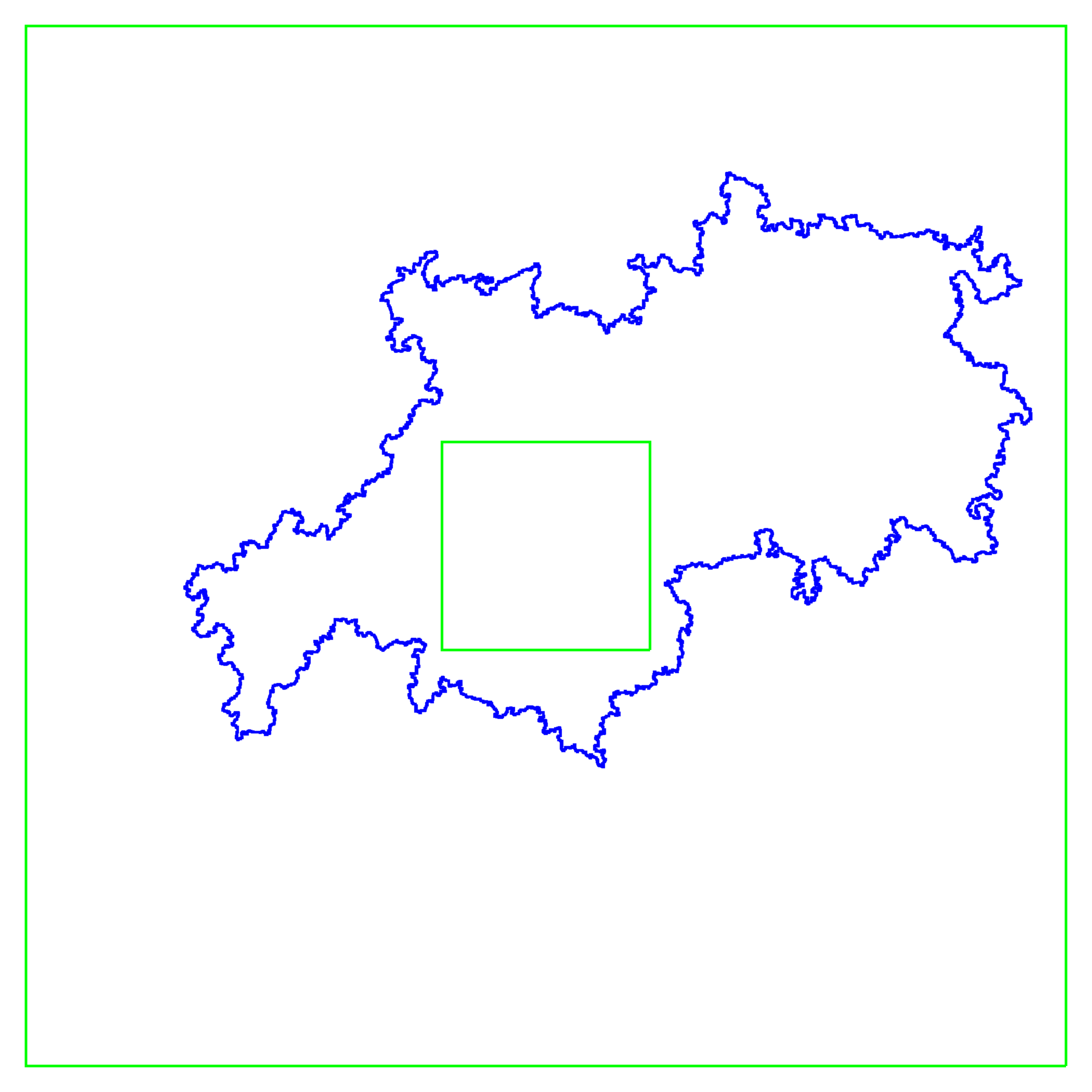}
\caption{A loop drawn according to $\mu^\#_A$. We first sampled the underlying UST using Wilson algorithm (\cite{Wi}), and explored it to find the dual loop. It is not possible to directly sample this loop using a modified Wilson algorithm, as in \cite{KasKen_RandomCurves}.}
\end{center}
\end{figure}

\begin{thm}\label{thm:cv-loop}
When the mesh size $\d$ goes to $0$, the random loop $\l^\d_{A^\d}$ converges in law (for the topology $\Top$ of uniform convergence up to reparametrization) towards a measure $\mu^\#_A$ on the set of loops ${\mc L}^\times(A)$ (the subset of topologically non-trivial loops in ${\mc L}(A)$, i.e. the set of simple loops drawn on the annulus $A$ that generate $\pi_1(A)$). Moreover, the family of measures $\mu^\#_.$ satisfies an explicit conformal covariance property (Proposition \ref{prop:covar-conf}).
\end{thm}

An important feature of the measure $\mu^\#_A$ is its invariance under conformal isomorphisms, and in particular its invariance under inversions (i.e. by the conformal isomorphisms of the annulus $A$ that switch the inner and outer boundaries).

This first statement of Theorem \ref{thm:cv-loop} will follow from the convergence of the exploration process (Definition \ref{def:explproc}) of the outer component of the wired UST. Consider a point $a$ on the outer boundary of $A$, and let $e^\d_{A^\d}$ be the counterclockwise exploration of the UST starting from $a^\d$, a lattice approximation of $a$.

\begin{thm2}
The exploration process $e^\d_{A^\d}$ converges in law (for the topology $\Top$) to a continuous process $e_A$.
\end{thm2}

To prove Theorem \ref{thm:cv-expl}, we will cut the exploration process into two parts.
Let us consider the first time $T$ when the trace of $e([0,T])$ disconnects the inner boundary from the point $a$. What happens after time $T$ is somewhat irrelevant for Theorem \ref{thm:cv-loop}, but is handed out to us along the way.

\begin{thm3}
After time T, $e_A$ behaves as a chordal $\SLE_8$ inside the remaining domain, headed towards $a$.
\end{thm3}

In order to understand the behavior of the exploration process up to time $T$, it is enough to understand it up to a certain family of stopping times $T_{c}^{\e}$ of supremum $T$. Consider a cut $c$, i.e. a smooth simple curve in $A$ connecting the two components of the boundary, and intersecting them orthogonally\footnote{Some regularity assumption on the boundary of $A$, and on the curve $c$ are needed for Theorems \ref{thm:cv-expl} and \ref{thm:cv-explbubble} to hold. It is indeed possible to construct a domain - the boundary of which is not a continuous curve - in which the continuous exploration process $e$ would not be a curve. Schramm proved certain estimates on the UST assuming $\mathcal{C}^1$ boundary, and we will follow him on this (see Theorem 11.1 of \cite{S0} and the remark that follows).}. Let $T_{c}^{\e}$ be the first hitting time of the $\e$-neighborhood of $c$ by the exploration process $e_A$.

\begin{thm4}
The law of $e_A$ stopped at all of the times $T_{c}^{\e}$ is enough to characterize the law of $e_A$ until the disconnection time $T$ of the point $a$ from the inner boundary of the annulus.
\end{thm4}

Consider a natural grid approximation $c^\d$ of the cut $c$. Up to time $T_{c^\d}^{\e}$, the exploration process $e^\d_{A^\d}$ is absolutely continuous with respect to the exploration process $e^\d_{A^\d\sm c^\d}$ of a wired UST in the simply-connected domain $A^\d\sm c^\d$.

Let us consider a curve $\g^\d$ (staying $\e$-away from $c^\d$) that traces the first steps of the exploration process of a spanning tree, and call $K^\d$ its image. The set $K^\d$ comes with a marked point on its boundary, namely the tip of the curve $\g^\d(t_0)$, and carries natural boundary conditions for the UST in the domain $A^\d \sm K^\d$: free on the counterclockwise arc $\B_{K^\d}=\arcperso{a^\d,\g^\d(t_0)}$, and wired on $\partial K^\d \sm \B_{K^\d}$. Indeed, the law of the UST restricted to $A^\d \sm K^\d$, conditioned on $\g^\d$ being the beginning of the exploration process, is a UST in $A^\d \sm K^\d$ with these boundary conditions.

The Radon-Nikodym derivative of the exploration processes can then be expressed as follows by definition:
$$
\frac{\P[e^\d_{A^\d}\mathrm{\ starts\ by\ }\g^\d]}{\P[e^\d_{A^\d\sm c^\d}\mathrm{\ starts\ by\ }\g^\d]} = \frac{\frac{\#T(A^\d\sm K^\d)}{\#T(A^\d)}}{\frac{\#T(A^\d\sm (K^\d\cup c^\d))}{\#T(A^\d\sm c^\d)}},
$$
where $\#T(\G)$ denotes the number of spanning trees on the graph $\G$.

We will need to rewrite this Radon-Nikodym derivative in a way more amenable to taking scaling limits. In order to do this, let us consider two cuts $d_1^\d$ and $d_2^\d$ that separate the cut $c^\d$ from $K^\d$. For simplicity, we will call $d^\d = d_1^\d\cup d_2^\d$ their union. We denote by $\Ha^{\alpha^\d\rightarrow\beta^\d}_{D^\d}$ the discrete operator of harmonic extension (Definition \ref{def:harmext-disc}) from a cut $\alpha^\d$ to another cut $\beta^\d$ in a domain $D^\d$. Boundary conditions for the harmonic extension (see Section \ref{sec:discharmobj}) correspond to the UST boundary conditions in the following way: Dirichlet corresponds to wired, and Neumann corresponds to free.
\begin{figure}[htb]\label{fig-Setup1}
\begin{center}
\includegraphics[width = 10cm]{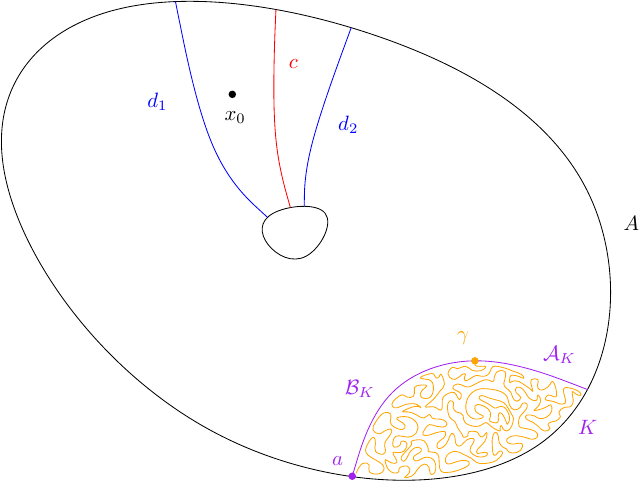}
\caption{The set-up of Lemma \ref{lem:DRN-expl}.}
\end{center}
\end{figure}

\begin{thm5}
We have:
$$
\frac{\P[e^\d_{A^\d}\mathrm{\ starts\ by\ }\g^\d]}{\P[e^\d_{A^\d\sm c^\d}\mathrm{\ starts\ by\ }\g^\d]}  = \frac{\det(\Id - \Ha^{c^\d\rightarrow d^\d}_{A^\d\sm{K^\d}} \circ \Ha^{d^\d\rightarrow c^\d}_{A^\d\sm{K^\d}})}{\det(\Id - \Ha^{c^\d\rightarrow d^\d}_{A^\d} \circ \Ha^{d^\d\rightarrow c^\d}_{A^\d})}.
$$
\end{thm5}

Hence, showing the convergence of $e^\d_{A^\d}$, can be reduced to two steps. First, showing the convergence of the reference process $e^\d_{A^\d\sm c^\d}$:

\begin{thm6}
The exploration process $e^\d_{A^\d\sm c^\d}$ converges in law (for the topology $\Top$) towards a counterclockwise $\SLE_8(2)$, i.e. an $\SLE_8$ aiming at its starting point.
\end{thm6}

And second, showing convergence of the following determinants:

\begin{thm7}
The determinants of the operators $\Id - \Ha^{c^\d\rightarrow d^\d}_{A^\d} \circ \Ha^{d^\d\rightarrow c^\d}_{A^\d\sm{d^\d}}$ and $\Id - \Ha^{c^\d\rightarrow d^\d}_{A^\d\sm{K^\d}} \circ \Ha^{d^\d\rightarrow c^\d}_{A^\d\sm{d^\d}}$ converge towards the determinants of their continuous counterparts.
\end{thm7}

The argument consisting in controlling the convergence of Radon-Nikodym derivatives in order to exploit directly the convergence to chordal SLE established in \cite{LSW_LERW} is somewhat novel and differs from the treatment of SLE convergence in multiply-connected domains of e.g. \cite{Zhan_LERWmult,Izy_Isingmult}.

The convergence of the discrete loop measures $\l^\d_{A^\d}$ towards a measure invariant under conformal isomorphisms (Theorem \ref{thm:cv-loop}) was already established by Adrien Kassel and Rick Kenyon (\cite{KasKen_RandomCurves}, Corollary 20). However, our approaches to this result are essentially disjoint, and complementary: Kassel and Kenyon characterize the limiting loop measure $\mu^\#_A$ via the law of its homotopy class in the annulus $A$ punctured at finitely many arbitrary interior points, relying in particular on difficult algebraic topology results of Fock and Goncharov. In this paper, we moreover investigate what becomes of the discrete restriction property in the continuous setting, which allows us to extend the family of measures to all Riemann surfaces.

The manuscript is organized as follows. We will start by fixing basic notations in Section \ref{def}, and discuss some combinatorial facts related to the UST, in particular how determinants of harmonic operators appear.  In Section \ref{precomp}, we will show tightness of the exploration process. We will then show convergence of the exploration process $e^\d_{A^\d\sm c^\d}$ in Section \ref{expl}, following the approach outlined above. In Section \ref{loop}, we will first prove Theorem \ref{thm:cv-loop} and look at the restriction properties of the family of loop measures $\mu^\#_A$. We will then extend this family to general Riemann surfaces, which is our main result, Theorem~\ref{thm:KSloop-2}.

\section{Background}\label{def}

\subsection{Riemann surfaces}

Let us first clarify the set-up of our paper regarding Riemann surfaces (we refer to \cite{FarKra_RS} for details).

\begin{itemize}

\item A \emph{Riemann surface} $\Si$ is a topological space that is modelled on the complex plane $\Co$: in particular, there is a notion of holomorphic functions on a Riemann surface.

\item The Riemann surfaces we consider will always be orientable and of finite topological type (i.e. the fundamental group $\pi_1(\Si)$ is finitely generated: we are excluding surfaces with infinitely many handles). In order to simplify notations, we do not assume our Riemann surfaces to be connected.

\item There is a unique compactification $\overline{\Si}$ of $\Si$ obtained by glueing a \emph{boundary} $\partial \Si$ (topologically a disjoint union of finitely many points and finitely many circles $\S^1$) such that any point of $\overline{\Si}$ has a neighborhood which is isomorphic (holomorphically) to a neighborhood of $0$ in either the complex plane $\Co$ or the upper half-plane $\H$.

\item A surface whose boundary contains no points is called \emph{puncture free}. The measures on loops we will consider give zero measure to the set of loops going through a predetermined point. Hence, we do not need to distinguish between two Riemann surfaces that are isomorphic up to finitely many punctures. Accordingly, we can and will assume all of our Riemann surfaces to be puncture free.

\item A Riemann surface which has an empty boundary (i.e. which is compact) is called \emph{closed}. An \emph{open} Riemann surface $\Si$ is a Riemann surface that is not compact. To such an open surface, we associate a closed Riemann surface $\widehat{\Si}$ - its \emph{Schottky double} - consisting of $\Si$ and a mirror copy of it, glued alongside their boundaries. For example, the Schottky double of a simply-connected domain is a sphere, and the Schottky double of a conformal annulus is a torus.

\item A (conformal) metric $g$ on a Riemann surface gives a notion of distance compatible with the complex structure. Given a (local) holomorphic isomorphism to $\Co$, a metric $g$ can be written as $e^{2\sigma}|dz|^2$ where $|dz|^2$ is the Euclidean metric on the plane. We call the metric smooth if the function $\sigma$ is smooth on the compactification $\overline{\Si}$, i.e. if partial derivatives of $\sigma$ of all orders exist and can be continuously extended to $\overline{\Si}$.

\item On an open Riemann surface $\Si$, a \emph{well-behaved} metric $g$ is a smooth metric such that each boundary component has a neighborhood which is isometric to a flat cylinder $[0,\e)\times(\R/2\pi\Z)$. A well-behaved metric $g$ naturally extends to a smooth metric $\hat{g}$ on the Schottky double $\widehat{\Si}$.

\item We say that a metric on a Riemann surface $\Si$ is \emph{normalized} if each connected component of $\Sigma$ has area $1$.

\item An important class of Riemann surfaces consists of \emph{domains}, i.e. open subsets of the complex plane. We call a domain \emph{smooth} if its boundary is a smooth (infinitely differentiable) curve. The Euclidean metric restricted to a domain is smooth (as defined above) if and only if the domain is. Moreover, note that the Euclidean metric is never well-behaved.

\end{itemize}

\subsection{Discretization of a continous set-up}
In the course of this paper, we will be interested in different discrete objects (living on planar graphs) that converge towards continuous objects defined on planar domains. These convergences are quite robust, and in particular would hold for any reasonable choice of graphs that approximate a planar domain. Let us describe how we will relate the discrete and continuous set-ups.

\subsubsection{The Carath\'eodory topology}

There is a natural topology on simply-connected domains of the complex plane with a marked interior point, called the Carath\'eodory topology. Let us first give a geometric description of it.

\begin{defn}\label{Car-top}
A sequence of simply-connected domains $(D_n,x_n)$ is said to Carath\'eodory-converge towards $(D_0,x_0)$ if
\begin{itemize}
\item $x_n\rightarrow x_0$.
\item any compactly-contained open subset of $D_0$ is included in $D_n$ for $n$ large enough.
\item any boundary point of $D_0$ is the limit of a sequence of boundary points of $D_n$.
\end{itemize}
\end{defn}

Taking the unit disc $(\Di,0)$ as a simply-connected domain of reference, we can rephrase convergence in the Carath\'edory topology.

\begin{thm}[Carath\'eodory's kernel theorem]\label{Car-ker-thm}
A sequence of marked domains $(D_n,x_n)$ Carath\'eodory-converges towards $(D_0,x_0)$ if and only if the uniformizing maps\footnote{We fix all degrees of freedom in the choice of the uniformizing map $\phi_n$ using the marked point, i.e. we require that $\phi_n(0)=x_n$, and for $\phi'_n(0)$ to be a positive real number.} $\phi_n : (\Di,0) \rightarrow (D_n,x_n)$ converge uniformly on compact subsets towards the uniformizing map  $\phi_0 : (\Di,0) \rightarrow (D_0,x_0)$.
\end{thm}

\begin{proof}See e.g. Theorem 1.8 in \cite{Pom_BoundaryConfMaps}. This theorem relies on the fact that it is possible to completely control the geometry with analytic data and vice versa (e.g. by using the Schwarz lemma and the Koebe quarter theorem).
\end{proof}

The Carath\'eodory topology can also be defined on the set of doubly-connected domains with a marked point, using the same geometric description. There is also an analytic point of view, even though the moduli space of doubly-connected domains is non-trivial. As reference domains, we can take the circular annuli $A(0,1,r)=\{z, 1<|z|<r\}$ with marked point $x_0\in (1,r)$ (and we will ask for the uniformizing maps to map inner boundary to inner boundary). A sequence of annular domains $A_n$ converges towards $A_0$ if their moduli $r_n$ converge towards the moduli $r_0$ of $A_0$, if the marked points converge, and if the uniformizing maps $A(0,1,r_n) \rightarrow A_n$ converge towards the uniformizing map $A(0,1,r_0) \rightarrow A_0$, uniformly on compact sets of $A(0,1,r_0)$.

The Carath\'eodory topology can be extended to sets of domains with a marked point $x_0$, carrying additional decoration, for example additional marked interior or boundary points, curves $c$ drawn inside the domain, or a hull\footnote{A hull $K$ is a compact subset $K \subset \overline{D}$, such that $D\sm K$ has the topology of $D$} $K$ not containing $x_0$. Marked points and drawn curves are compared on reference domains via the uniformization maps (using the topology $\Top$ of supremum norm up to reparametrization to compare curves). We compare hulls $K$ using the Carath\'eodory topology for $(D\sm K,x_0)$.

\subsubsection{Domain approximation}

We call a discrete domain of mesh size $\delta$ a connected union of faces of the lattice $\d \Z^2$. We can see them alternatively as open subsets of $\mathbb{C}$ or as graphs. They can also carry decorations living on the lattice $\d \Z^2$.

\begin{defn}\label{disc-approx}A sequence $(D_n, x_n)$ of discrete decorated domains of mesh size $\d_n \rightarrow 0$ is said to be an approximation of a decorated domain $(D, x_0)$ if $(D_n, x_n)$ Carath\'eodory-converges towards $(D, x_0)$.
\end{defn}

There is a natural choice of approximation of mesh size $\d$ for a domain with a marked point, which will allow us to state uniform convergence results.

\begin{defn}\label{natural-approx}
Let us consider a domain $D$, with a marked point $x_0 \in D$. The natural approximation $D^\d$ of $D$ at mesh size $\d$ is the largest discrete domain of mesh size $\d$ included in $D$ and containing $x_0$. In other words, $D^\d$ is the connected component containing $x_0$ of the set of all faces of the graph $\d \Z^2$ that are sitting inside $D$. We approximate marked points by taking the closest\footnote{Once $x_0^\d$ has been chosen, proximity should be measured after having mapped $D$ to its reference domain.} point of $\d \Z^2$. To approximate a simple curve $c$, we take $c^\d$ to be one of the two simple curves living on $\d \Z^2$ that stay the closest possible to $c$ without intersecting it. The natural approximation $K^\delta$ of a hull $K\subset \overline{D}$ is the complement in $D^\delta$ of the natural approximation of $(D\sm K,x_0)$.
\end{defn}

\subsection{Harmonic analysis}

\subsubsection{Discrete harmonic objects}\label{sec:discharmobj}

Consider $\G$ a finite subgraph of $\Z^2$. Let $F$ be a function defined on the vertices of $\G$. We define the discrete partial derivative $\partial F$ on oriented edges $e=xy$ of $\G$ as the difference of the values taken by $F$ at the endpoints: $F(y)-F(x)$.

Let us consider a subset $\partial \G$ of vertices of $\G$ that we call boundary (the complement $\G \sm \partial\G$ of which we call interior vertices), and let us split this boundary in two parts: a Dirichlet boundary $\partial \G_D$, and a Neumann boundary\footnote{We need to have one Neumann vertex for each edge connecting the Neumann boundary to the interior of $\G$. To achieve this, one can modify the graph $\G$ by splitting each Neumann vertex in as many vertices as they are edges connecting it to the interior of $\G$.} $\partial \G_N$. We say that a function $F$ has Dirichlet boundary condition given by $f$ on $\partial \G_D$ if $F=f$ there. When no function $f$ is specified, we always imply $f=0$. We say that $F$ has Neumann boundary condition on $\partial \G_N$ if its derivative $\partial F$ is $0$ on all edges connecting $\partial \G_N$ to the interior of $\G$.

\begin{defn}\label{def:Laplacian}
Let $F$ be a function defined on the interior of $\G$ (and naturally extended to the boundary\footnote{So that it has Dirichlet boundary conditions on $\partial \G_D$ and Neumann boundary conditions on $\partial \G_N$.}). On interior vertices of $\G$, we can define the Laplacian of $F$ on the interior of $\G$ to be
$$
\D F(z) = F(z+1) + F(z+i) + F(z-1) + F(z-i) - 4F(z).
$$
A discrete function $F$ is said to be harmonic if $\D F = 0$.
\end{defn}

Let us now define some harmonic objects on $\G$. For any vertex $x \in \G$, the harmonic measure $\mu_x(.)$ is a probability measure on $\partial \G_D$, or equivalently, a collection of non-negative numbers $(\mu_x(\{y\}))_{y\in\partial\G_D}$, summing to $1$.

\begin{defn}\label{def:disc-Harmonic-measure}
The function $x \mapsto \mu_x(\{y\})$ is the unique harmonic function on $\G$ with Dirichlet boundary condition $0$ on $\partial \G_D \sm\{y\}$, $1$ on $\{y\}$, and Neumann boundary condition on $\partial\G_N$.
\end{defn}

Alternatively, $\mu_x(.)$ is the exit measure of a simple random walk starting from $x$, ``reflected'' on $\partial\G_N$ and stopped upon hitting $\partial \G_D$.

We now fix a distinguished vertex $x_0$ in $\G$.

\begin{defn}\label{def:disc-Poisson-kernel}
For $x\in\G$ and $y \in \partial\G_D$, the Poisson kernel normalized at $x_0$ is the quantity
$$
P^\G_{x_0}(y,x) = \frac{\mu_x(\{y\})}{\mu_{x_0}(\{y\})}.
$$
\end{defn}

Finally, let us consider two disjoint discrete cuts $\alpha^\d$ and $\beta^\d$ in a domain $D^\d$.

\begin{defn}\label{def:harmext-disc}
Given a function $f$ defined on the cut $\a^\d$, we can extend it to a function $F(x)=\sum \mu_x(\{y\}) f(y)$, which is harmonic on the domain $D^\d\setminus\alpha^\d$. We then denote by $\Ha^{\alpha^\d\rightarrow\beta^\d}_{D^\d} : f \mapsto F_{|\beta^\d}$ the operator of harmonic extension from $\alpha^\d$ to $\beta^\d$ in the domain $D^\d$ that maps the function $f$ to the restriction of its harmonic extension $F$ to the cut $\beta^\d$.
\end{defn}

\subsubsection{Continuous harmonic objects}

Consider a Riemann surface equipped with a conformal metric $(\Si,g)$ - an important particular case of this set-up being a domain $(D,|dz|^2)$ of the complex plane equipped with the Euclidean metric. We split the boundary of $\Si$ in a Dirichlet and a Neumann part, $\partial \Si_D$ and $\partial \Si_N$ (such that each one is a finite union of boundary arcs). We are going to consider smooth functions on $\Si$ that continuously extend to the boundary (except possibly at a finite number of points). If the metric is given in local coordinates by $g(z) |dz|^2$, we define the Laplacian to be $\D^g= g^{-1}(z) \D$ , where $\D = \partial^2_{xx} +\partial^2_{yy}$ is the positive Euclidean Laplacian. A harmonic function is a real-valued function $F$ such that $\D^g F=0$.

Until further notice, we now work on a domain $D$ of the complex plane equipped with the Euclidean metric.

\begin{defn}\label{def:Harmonic-conjugate}
Let $F$ be a harmonic function on $D$. Its harmonic conjugate\footnote{The real and imaginary parts of a holomorphic function are harmonic conjugates.} $G$ is locally defined up to an additive constant, as the function that satisfies $\partial_x G = - \partial_y F$ and $\partial_y G = \partial_x F$.
\end{defn}

This allows to make sense of Neumann boundary conditions for a harmonic function $F$, even if the boundary $\partial D_N$ is not smooth (as in \cite{LSW_LERW}): we can require its harmonic conjugate $G$ to extend continuously to, and be constant on (connected components of) $\partial D_N$.

Let us now define some harmonic objects on $D$. For any point $x \in D$, the harmonic measure $\mu_x(.)$ is a probability measure on $\partial D_D$:

\begin{defn}\label{def:cont-Harmonic-measure}
The harmonic measure $\mu_x(.)$ is the exit measure of planar Brownian motion starting at $x$, reflected normally\footnote{The trace of planar Brownian motion being conformally invariant, normally reflected Brownian motion can be defined (up to time-reparametrization) when the boundary is not smooth, by uniformizing the domain.} on $\partial D_N$ and stopped on $\partial D_D$.
\end{defn}

Alternatively, if $I$ is a subarc of $\partial D_D$, $x \mapsto \mu_x(I)$ is the unique bounded harmonic function on $D$ with Dirichlet boundary condition $0$ on $\partial D_D \sm I$, $1$ on $I$, and Neumann boundary condition on $\partial D_N$.

Let us now fix a point $x_0$ in $D$.

\begin{defn}\label{def:cont-Poisson-kernel}
For $(y,x) \in \partial D_D \times D$, the Poisson kernel normalized at $x_0$ is the Radon-Nikodym derivative:
$$
P^D_{x_0}(y,x)=\frac{\text{\normalfont d}\mu_{x}}{\text{\normalfont d}\mu_{x_0}}(y).
$$
\end{defn}

The function $P^D_{x_0}(y,.)$ is harmonic. It is actually the kernel for the Poisson problem: given a continuous function $f(y)$ on $\partial D_D$, the unique bounded harmonic function $F$ on $D$ that has Dirichlet boundary condition $f$ on $\partial D_D$ and Neumann boundary condition on $\partial D_N$ is given by
$$
F(x) = P_{x_0}^D f= \int_{\partial D_D}{P_{x_0}^D(y,x)f(y)\mu_{x_0}(dy)}.
$$
For example, in the upper half-plane $\mathbb{H}$ with full Dirichlet boundary conditions, the Poisson kernel is given by $P_i(0,z)=-\Im(1/z)$.

Finally, let us consider two disjoint cuts $\alpha$ and $\beta$ in a domain $D$. To a continuous function $f$ defined on the cut $\a$, we can associate a function $\Ha^{\alpha\rightarrow\beta}_{D}(f)$ on the cut $\beta$, by first extending $f$ to a harmonic function $F=P_{x_0}^{D\setminus\alpha} f$ on $D\setminus\alpha$ and then restricting $F$ to $\beta$.

\begin{defn}\label{def:harmext}
We denote by $\Ha^{\alpha\rightarrow\beta}_{D} : f \mapsto F_{|\beta}$ the operator of harmonic extension from $\alpha$ to $\beta$ in the domain $D$.

\end{defn}

\subsubsection{Harmonic analysis toolbox}

Let $B_x^\d(r)$ be the approximation of mesh size $\d$ of the ball of radius $r$ centered at a point $x$.

\begin{lem}[Harnack inequality]\label{lem:harnack}
There is an absolute constant $c$ such that for any non-negative discrete harmonic function $f$ defined on the ball $B_x^\d(R)$, and for any point $y\in B_x^\d(r)\subset B_x^\d(R)$, with $r<R/2$, we can bound the increments of $f$:
$$
|f(x)-f(y)| \leq c \frac{r}{R} f(x).
$$

\end{lem}

\begin{proof}
See e.g. Proposition 2.7 (ii) in \cite{SmiChe_isoradial} for a stronger estimate.
\end{proof}

\begin{lem}[Beurling estimate]\label{lem:beurling}
Consider, on a discrete domain $D$ of mesh size $\d$, a harmonic function $f$ bounded by $1$ that has $0$ Dirichlet boundary conditions on some boundary arc $\A$. For any point $x\in D$, call $\e$ its distance to $\A$, and $d$ its distance to the rest of the boundary $\partial D \sm \A$. There is an absolute constant $\beta>0$ such that $f(x) = O((\e/d)^\beta)$, uniformly on Carath\'eodory-compact sets of decorated domains, and in $\d$. 
\end{lem}

\begin{proof}
See e.g. Proposition 2.11 in \cite{SmiChe_isoradial}.
\end{proof}

\begin{lem}\label{lem:poisson-kernel-bounded}
The discrete Poisson kernel is uniformly bounded away from its boundary singularity, on Carath\'eodory-compact sets of decorated domains.
\end{lem}

\begin{proof}
See the proof of Theorem 3.13 in \cite{SmiChe_isoradial} for the case of simply-connected domains with full Dirichlet boundary. The core of the argument is a local study of the singularity, and carries through for doubly-connected domains, as well as when there is a non-trivial Neumann boundary.
\end{proof}

\begin{lem}[\cite{LSW_LERW}, Proposition 4.2]\label{lem:harm-meas-conv}
Let us consider a simply-connected domain $(D,x_0)$, with two disjoint boundary arcs $\A_1$ and $\A_2$ that are not both empty. Let $\B = \partial D \sm (\A_1\cup \A_2)$. We also consider the natural approximation of this setting.
The discrete harmonic measure of $\A^\d_1$ with Neumann boundary conditions on $\B^\d$ seen from $x^\d_0$ converges towards its continuous counterpart, uniformly on Carath\'eodory-compact sets of such decorated domains.
\end{lem}

\begin{rem}\label{rem:harm-meas-conv-bound}
Using the Beurling estimate, we can see that the above convergence is actually uniform in $x_0 \in D$ staying away from $\partial D \sm \A_2$
\end{rem}

\begin{rem}\label{rem:harm-meas-conv-2-conn}
The result of Lemma \ref{lem:harm-meas-conv} can be easily extended to doubly-connected domains in two special cases: when $\B$ either consists of a whole boundary component, or when $\B$ is empty and $\A_2$ is a union of boundary arcs. Indeed, when $\B$ is empty, the arguments given in the proof of \cite{SmiChe_isoradial},Theorem 3.12 will apply. When $\B$ is a whole boundary component, harmonic conjugates are single-valued, and the proof in \cite{LSW_LERW} carries through.
\end{rem}

\subsection{Determinants and loop measures}

For our purposes, it will be useful to rewrite some expressions involving the determinant of the discrete Laplacian, in a way that easily allows to take scaling limits.
These determinants are related to probabilistic objects, namely loop measures, that we will use only peripherally in this paper (we refer to \cite{LeJan_Loops} and \cite{LW} for precise definitions of these loop measures).

\subsubsection{Loop measures}

To a symmetric Markov process on a finite space $\G$ (e.g. the simple random walk on a finite subgraph of $\Z^2$, with mixed stopped/reflected boundary conditions), one can associate a natural measure $\mu^{\lo}$ on loops (closed paths) living on $\G$ (see Section 2.1 in \cite{LeJan_Loops}).

Let $\G$ be a subgraph of $\Z^2$, and consider the loop measure $\mu^{\lo}_{\G}$ associated to the simple random walk on $\G$. We have the following expression for the total mass of loops.

\begin{lem}[\cite{LeJan_Loops}, Equation 2.5]\label{lem:det-Lap-loop-meas}
$|\mu^{\lo}_{\G}| = -\log\det(\D_{\G})$.
\end{lem}

These loop measures have a natural equivalent in the continuous setting: a loop measure $\mu_D^{\lo}$ associated to Brownian motion can be defined in any domain $D$ of the complex plane (see Section 4 of \cite{LW}).

We will now discuss two different notions of determinants for certain infinite-dimensional linear operators.

\subsubsection{Fredholm determinant}

Let $\T$ be an integral kernel operator on the function space $\mathcal{L}^2([0,a],dx)$, i.e. an endomorphism of this function space of the form $\T f(y)=\int_{[0,a]} \T(y,x)f(x)dx$ for some bicontinuous function $\T$.

\begin{defn}
The Fredholm determinant of $\Id + \T$ is:
$$
{\det}_F(\Id+\T) = \sum_n \frac{1}{n!} \int_{x_1,\cdots,x_n} \det\left([\T(x_i,x_j)]_{1\leq i,j\leq n}\right) \ dx_1 \cdots dx_n.
$$
\end{defn}

In Chapter 3 of \cite{SimTrace}, it is explained why this gives a reasonable generalization of the notion of determinant (e.g. it encodes information on the invertibility of the operator $\Id+\T$).

\subsubsection{$\zeta$-regularized determinant}

Let us first give an overview of $\zeta$-regularization before getting into details.

Suppose we have a countable Hilbert basis of eigenvectors corresponding to regularly increasing positive eigenvalues $\lambda_i$ of an operator $L$ (e.g. $L$ is the positive Laplacian on a Riemann surface $(\Si,g)$ with a smooth metric and with a non-trivial Dirichlet boundary). We can then build the zeta function $\zeta_L(s) = \sum_i \lambda_i^{-s}$. This series converges when $\Re(s)$ is large enough to a function that admits a meromorphic extension to the whole complex plane, which is moreover holomorphic near the origin.

\begin{defn}\label{def:zetadet}
The $\zeta$-regularized determinant of $L$ is $\det_\z(L)= e^{-\z_L'(0)}$.
\end{defn}

Note that this definition gives the usual determinant when $L$ is a finite-dimensional operator.
On the other hand, the quantity $-\log\det_\z(\D_{D})$ can be interpreted (by analogy with Lemma \ref{lem:det-Lap-loop-meas}) as a regularization of the total mass of Brownian loops on the domain $D$.

\begin{rem}\label{def:zetadetprime}
If $L$ has a zero eigenvalue (e.g. $L$ is the positive Laplacian on a Riemann surface $(\Si,g)$ with a smooth metric and without Dirichlet boundary), we can define $\det'_\z(L)$ in a similar way by ignoring the zero eigenvalue in the series defining the zeta function $\zeta_L(s)$.
\end{rem}

As there can be no ambiguity, we will in the following indiscriminately use $\det_\z$ to denote either $\det_\z$ or $\det'_\z$, depending on whether the surface under consideration has a non-trivial Dirichlet boundary.

Let us now give some more details (we assume that $L$ is some Laplacian operator). We mainly refer to \cite{BerGetVer_HeatKernels}. It is actually easier to define $\zeta_L(s)$ as the Mellin transform of the trace of the heat kernel\footnote{The heat kernel $e^{-tL}$ is the fundamental solution to the heat equation $\frac{\partial}{\partial t} + L = 0$. The heat kernel of the Laplacian is trace class (\cite{BerGetVer_HeatKernels}, Proposition 2.32), i.e. it is sufficiently well-behaved so that its trace can be defined unambiguously.} $e^{-tL}$, namely: $\zeta_L(s)=\Me[\Tr(e^{-tL})](s)$ where the Mellin transform of a function $f$ defined on the positive real line is given by the integral formula:
$$
\Me[f](s)=\frac{1}{\Gamma(s)}\int_0^{\infty} f(t) t^{s-1} dt.
$$

To relate this to the function $\zeta_L$ discussed above, note that the Mellin transform of an exponential is given by $\Me[e^{-t\lambda}](s)=\lambda^{-s}$, by definition of the $\Gamma$ function. Summing over eigenvalues of the Laplacian (formally) gives:
$$
\Me[\Tr(e^{-tL})](s) = \sum_i \Me[e^{-t\lambda_i}](s) = \sum_i \lambda_i^{-s} = \zeta_L(s).
$$
If $s$ is of real part large enough, Weyl's asymptotics for the eigenvalues of the Laplacian (\cite{BerGetVer_HeatKernels}, Corollary 2.43) ensures that the above series converges fast enough, so that the computation rigorously holds.
If $L$ admits zero as an eigenvalue, we only want to sum over positive eigenvalues, and so we should correspondingly consider the Mellin transform of the trace of $P_{(0,\infty)}e^{-tL}$ where $P_{(0,\infty)}$ is the orthogonal projection on the space generated by the eigenfunctions corresponding to positive eigenvalues.

General properties of the Mellin transform (\cite{BerGetVer_HeatKernels}, Lemma 9.34) show that if a function $f$ has a nice short-time asymptotics (i.e. if there is an integer $n$ such that $f$ behaves for small times $t$ like $\sum_{k\geq -n} f_k t^{k/2}$) and is also well-behaved at $\infty$ (i.e. $f$ decays exponentially fast), the Mellin transform $\Me[f](s)$, a priori only well-defined for numbers $s$ of large enough real part, actually extends to a meromorphic function of the whole plane, which is moreover holomorphic at $0$.

We can hence ensure our definitions \ref{def:zetadet} and \ref{def:zetadetprime} make sense if we can show exponential decay of the trace of the heat kernel of the Laplacian, and compute its short-time asymptotics (in two dimensions, we have a short-time expansion with $n=2$).

Proposition 2.37 in \cite{BerGetVer_HeatKernels} gives smoothness (in particular measurability) of $\Tr(e^{-tL})$ and exponential decay of the trace of the heat kernel (restricted to the positive eigenspaces) for large time. The same argument extends to manifolds with boundary.

Short-time asymptotics of the heat kernel associated to the Laplacian (on a manifold with or without boundary) are nicely discussed in \cite{Grieser_HeatKernel}. The Minakshisundaram-Pleijel short-time expansion (for manifolds without boundary) can also be found as Proposition 2.47 in \cite{BerGetVer_HeatKernels}. For the short-time expansion of the heat kernel on a surface with Dirichlet boundary, we also refer to the original paper of McKean and Singer \cite{McKeanSinger_Laplacian}.

\subsubsection{Determinantal identities}
We will now state an identity between determinants of harmonic operators and masses of loops that will be useful later on.

Let $D$ be a bounded domain of the complex plane, and let $K_1$ and $K_2$ be two disjoint connected compact subsets of its closure. We moreover consider a (possibly empty) boundary arc $\B$ that is disjoint from the two compact sets $K_1$ and $K_2$, and call $\A = \partial D\sm \B$ the rest of the boundary of $D$.
\begin{figure}[htb]\label{fig-Setup3}
\begin{center}
\includegraphics[width = 6cm]{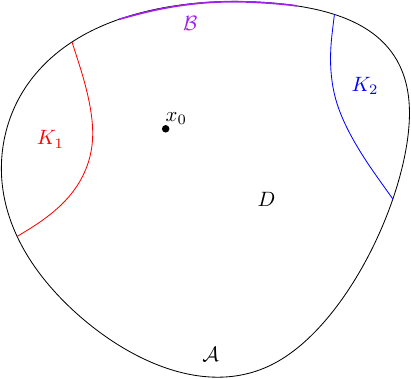}
\caption{The set-up of Lemma \ref{lem:DRN-expl}.}
\end{center}
\end{figure}
We put Dirichlet boundary conditions on $\partial K_1$, $\partial K_2$ and $\A$ and Neumann boundary conditions on $\B$.

Let us also consider a discrete approximation at mesh size $\d$ of this setting.

\begin{lem}\label{lem:ratio-det}
We have the following discrete identities:
$$
\frac{\det(\D_{D^\d})\det(\D_{D^\d \sm (K_1^\d \cup K_2^\d)})}{\det(\D_{D^\d\sm K_1^\d})\det(\D_{D^\d\sm K_2^\d})} = e^{-\mu^{\lo}_{D^\d}\{l|l\cap K^\d_1\neq\emptyset, l\cap K^\d_2\neq\emptyset\}} = \det(\Id - \Ha^{\partial K_1^\d\rightarrow \partial K_2^\d}_{D^\d} \circ \Ha^{\partial K_2^\d\rightarrow \partial K_1^\d}_{D^\d}).
$$
The following continuous equalities hold (where $\zeta$-regularized determinants are defined only if the domains are smooth):
$$
\frac{\det_\z(\D_{D})\det_\z(\D_{D \sm (K_1 \cup K_2)})}{\det_\z(\D_{D\sm K_1})\det_\z(\D_{D\sm K_2})} = e^{-\mu^{\lo}_{D}\{l|l\cap K_1\neq\emptyset, l\cap K_2\neq\emptyset\}} = {\det}_{F}(\Id - \Ha^{\partial K_1\rightarrow \partial K_2}_{D} \circ \Ha^{\partial K_2\rightarrow \partial K_1}_{D}).
$$
\end{lem}

\begin{proof}The equalities in the continuous setting are stated as Proposition 2.1 and 2.2 in \cite{Dub_SLEGFF} (the proof carries through if there is some Neumann boundary). The second equality in the discrete setting is proved similarly as its continuous counterpart, whereas the first one is a consequence of Lemma \ref{lem:det-Lap-loop-meas}.
\end{proof} 

\begin{rem}\label{rem:ratio-det-non-smooth}
Note that the notion of harmonic function does not depend on the underlying metric. In particular, the quantity $\det_F(\Id - \Ha^{\partial K_1\rightarrow \partial K_2}_{D} \circ \Ha^{\partial K_2\rightarrow \partial K_1}_{D})$ only depends on the complex structure, and not on the metric. Hence if $D$ is any domain of the plane, and $K_1$ and $K_2$ are smooth subdomains of $D$\footnote{Meaning that if we conformally map $D$ to a smooth domain, $K_1$ and $K_2$ are mapped to compact sets with smooth boundary.}, then for any smooth metric $g$ on $D$, we have $\frac{\det_\z(\D^g_{D})\det_\z(\D^g_{D \sm (K_1 \cup K_2)})}{\det_\z(\D^g_{D\sm K_1})\det_\z(\D^g_{D\sm K_2})} = \det_F(\Id - \Ha^{\partial K_1\rightarrow \partial K_2}_{D} \circ \Ha^{\partial K_2\rightarrow \partial K_1}_{D})$.
\end{rem}

\subsection{The determinant line bundle}\label{subsec-cocycle}

Let us now define some objects (following Kontsevich and Suhov \cite{KontSuh}) that will allow us to extend our family of loop measures $\mu^\#_A$ to general Riemann surfaces.

\subsubsection{Real line bundles}

We briefly recall some facts about real line bundles.

\begin{itemize}

\item Given a topological space $X$, a topological \emph{real line bundle} $L$ over the base space $X$ is the data of a (continuously varying) one-dimensional real vector space $l(x)$ for every point $x$ - which is called the fiber or the line above $x$. The \emph{trivial line bundle} over $X$ is the space $X\times\R$ where all fibers are canonically identified to the vector space $\R$. Line bundles can be interesting because of their global topology: on $\S^1$ for example, we can construct a line bundle homeomorphic to a Moebius strip.

We call a line bundle \emph{oriented} if its fibers carry a (locally consistent) orientation.

\item This provides a way to generalize functions, by looking at \emph{sections} of a line bundle. A section is the data for any $x\in X$ of a (continuously varying) point $s(x)\in l(x)$. Sections of the trivial line bundle are canonically identified with continuous functions over $X$.

\item We call a line bundle $L$ \emph{trivializable} if there exists an isomorphism\footnote{Isomorphisms of line bundles restrict to the identity on the base space $X\times\{0\}$ and are linear on the fibers $l(x)$.} $\phi$ to the trivial line bundle $X\times\R$. The data of such a \emph{trivialization} $\phi$ is equivalent to the data of a non-vanishing section of $L$. Indeed the trivial line bundle $X\times\R$ has a canonical non-vanishing section: the constant function $s(x)=1$, which can be pushed to a non-vanishing section of $L$ via any isomorphism. Conversely, a non-vanishing section of $L$ gives a trivialization. Identifying a trivializable line bundle as the space $X\times\R$ is usually not canonical.

\item Given an oriented real line $l$, we can define its $c$-th power $l^{\otimes c}$ for any real number $c$ in the following way. For $c=0$, we set $l^{\otimes 0}=\R$. For a non-zero real number $c$, we define the positive half-line of the space $l^{\otimes c}$ as the set of formal vectors $v^{\otimes c}$ for positive $v\in l$, equipped with the scalar multiplication $\lambda v^{\otimes c} := (\lambda^{1/c} v)^{\otimes c}$ for positive $\lambda$, and with the unique additive structure compatible with this scalar multiplication. One can check that, for integer powers $c=n$, this is consistent with the usual $n$-th tensor powers. Moreover, the spaces $l^{\otimes c}$ and $l^{\otimes -c}$ are canonically dual (there exists a pairing such that $v^{\otimes c}\cdot v^{\otimes -c}=1$ for any non-zero $v\in l$). For a trivializable oriented line bundle $L$ with fibers $l(x)$, we define its $c$-th power $L^{\otimes c}$ as the line bundle whose fibers are the lines $l(x)^{\otimes c}$.

\item A measure is dual to functions on $X$, or in other words dual to sections of the trivial line bundle $X\times\R$. Given a trivializable line bundle $L$, we call $L^*$-\emph{valued measures} the objects dual to sections of the line bundle $L$. Given a non-vanishing section $s$ of the dual line bundle $L^*$, any $L^*$-valued measure $\mu$ can be written $s\mu_s$ for some (scalar, signed) measure $\mu_s$ on the base space $X$. Indeed, note that $s$ pairs pointwise with any section of $L$ to give a function, that can then be integrated against $\mu_s$

\end{itemize}

In the following, our base space will be the set of simple loops $X=\L(\Sigma)$ on a Riemann surface $\Sigma$. We will describe a trivializable oriented line bundle $L=|\Det|_\Sigma$ on $X$, called the determinant line bundle\footnote{We follow the notations and terminology of \cite{KontSuh}. The standard determinant line bundle $|\det|$ (implicitly defined in Section \ref{sec:detlinRS}) is a real line bundle with base space the moduli space of Riemann surfaces.}. Any embedding of Riemann surfaces $\Sigma'\hookrightarrow\Sigma$ will provide a map $|\Det|_\Sigma' \hookrightarrow |\Det|_\Sigma$ (see Proposition \ref{prop:incldetbund}). The geometric interest of the family of determinant line bundles lies in the fact that, even if each one of them is trivializable, the family as a whole cannot be trivialized in a way consistent with all possible embeddings of Riemann surfaces.

Moreover (see Section \ref{from-ann-to-rs}), we will construct a measure $\mu_\Sigma$ on the space of simple loops $\L(\Sigma)$. We will argue that a more natural object than the family $\mu_\Sigma$ is the family of twisted measures $\lambda_\Sigma = s_\Sigma^c \mu_\Sigma$ where $s_\Sigma^c$ is a (non-vanishing) section of the $c$-th power of the determinant line bundle over $X$, i.e $\lambda_\Sigma$ is an $|\Det|_\Sigma^{\otimes c}$-valued measure. Any choice of trivialization of the line bundle $|\Det|_\Sigma$ (e.g. $s_\Sigma$) then gives a (scalar) measure on the base space $\L(\Sigma)$.

\subsubsection{The determinant line associated to a Riemann surface}\label{sec:detlinRS}

To a Riemann surface $\Sigma$, we associate an oriented line $|\det|_\Sigma$, the vector space generated by formal vectors $[g]$ associated to smooth (and well-behaved - if the surface is open) metrics compatible with the complex structure, and quotiented by the relations\footnote{Remark 2.1 and Lemma 2.1 in \cite{KontSuh} show why the quotient is a half-line and not a point.}:
$$
[e^{2\sigma}g]=\exp(S_L(g,\sigma))[g],
$$
where $S_L$ is the Liouville action ($K$ denotes the scalar curvature and $dA$ is the area form):
$$
S_L(g,\sigma)=-\frac{1}{12\pi}\int_\Sigma\left(\frac 12\|\nabla_g\sigma\|^2+K_g\sigma\right)dA_g.
$$
Note that the map $g\mapsto [g]$ is not necessarily homogeneous: on a closed surface, by the Gauss-Bonnet formula, $[\lambda g]=\lambda^p[g]$ where the power $p$ is given by $3p=\text{genus}(\Sigma)-1$.

A linear form on $|\det|_\Sigma$ is an element $\psi\in|\det|_\Sigma^{\otimes-1}$, i.e. is such that $\psi([e^{2\sigma} g])=\exp(S_L(g,\sigma))\psi([g])$, and it thus may be identified with a functional (also denoted $\psi$) defined on the space of metrics satisfying the anomaly formula
$$
\psi(e^{2\sigma} g)=\exp(S_L(g,\sigma))\psi(g).
$$
From this representation of $|\det|_\Sigma^{\otimes -1}$ (and by definition of fractional powers), we see that an element of the fractional power of the determinant line $|\det|_\Sigma^{\otimes c}=\left[|\det|_\Sigma^{\otimes -1}\right]^{\otimes -c}$ may be identified with a functional $\psi$ such that
\begin{equation}\label{eq:detlinpow}
\psi(e^{2\sigma} g)=\exp(-cS_L(g,\sigma))\psi(g).
\end{equation}


\subsubsection{The Polyakov-Alvarez conformal anomaly formula}\label{polalv}

On a Riemann surface $\Si$, consider two metrics agreeing with the complex structure, $g$ and $g'=e^{2\s}g$.

\begin{thm}[Polyakov-Alvarez conformal anomaly formula, \cite{OsgPhiSar_DetLaplacians}, Equations 1.13 and 1.15]

If $\Si$ is a Riemann surface without boundary,
$$
\log \frac{\det'_\z(\D^{g'})}{\det'_\z(\D^{g})} = 2 S_L(g,\sigma) + \log \frac{\mathrm{Area}_{g'}(\Sigma)}{\mathrm{Area}_g(\Sigma)}.
$$
If $\Si$ has a non-trivial boundary with full Dirichlet boundary conditions,
$$
\log \frac{\det_\z(\D^{g'})}{\det_\z(\D^{g})} = 2 S_L(g,\sigma)  -\frac{1}{6\pi}\int_{\partial \Si} k_g \s ds_g -\frac{1}{4\pi} \int_{\partial \Si} \partial_n\s ds_g,
$$
where $k$ is the curvature of the boundary, $\partial_n$ is the outer normal derivative, and $ds$ is the element of arclength on $\partial \Si$.
\end{thm}

Note that both boundary integrands vanish if the metrics $g$ and $g'$ are well-behaved (indeed, $\partial_n\s=0$, as can be seen by symmetry on the Schottky double of $\Si$).

\begin{rem}\label{rem:pol-alv-neum}
We can deduce from these formulas a similar explicit formula when $(\Si,g)$ has some Neumann boundary components. Indeed, consider the doubling of $\Si$ consisting of $\Si$ and its mirror copy, glued alongside their Neumann boundaries, which we denote $\widehat{\Si}$ by abuse of notation. The metric $g$ being well-behaved near the Neumann boundary of $\Si$, it extends to a smooth metric $\widehat{g}$ on the doubled surface $\widehat{\Si}$. Let us call $\Si^D$ the surface $\Si$ with all its boundary conditions changed to Dirichlet. Via symmetrization (resp. antisymmetrization), Laplacian eigenfunctions on $\widehat{\Si}$ are in correspondence with Laplacian eigenfunctions on $\Si$ (resp. $\Si^D$). As a consequence, the spectrum of the Laplacian on $\widehat{\Si}$ is the union of the Laplacian spectra on $\Si$ and $\Si^D$. In other words, $\log\det_\z(\D^{\widehat{g}}_{\widehat{\Si}})=\log\det_\z(\D^g_\Si)+\log\det_\z(\D^g_{\Si^D})$.
\end{rem}

On a closed Riemann surface $\Si$, the conformal anomaly formula can be rephrased (see (\ref{eq:detlinpow})), by saying that the functional
\begin{equation}\label{eq:psi1}
\psi_{\Si}(g)=\frac{{\det}'_\zeta(\D^g)}{\mathrm{Area}_g(\Sigma)}
\end{equation}
is an element of the line $|{\det}|_\Sigma^{\otimes-2}$. Incidentally, the conformal anomaly formula shows that $S_L(g,\sigma)=0$ whenever $(\Sigma,g)$ and $(\Sigma,e^{2\sigma}g)$ are isometric. Moreover, on an open connected surface $\Si$, if $\D_{\Si^D}$ and $\D_{\Si^N}$ designate the Laplacian with Dirichlet (resp. Neumann) boundary conditions, the functionals $\psi_{\Si^D}$ and $\psi_{\Si^N}$ given by
\begin{equation}\label{eq:psi2}
\psi_{\Si^D}(g)={\det}_\zeta(\D_{\Si^D}^g)\ \mathrm{and}\ \psi_{\Si^N}(g)=\frac{{\det}'_\zeta(\D_{\Si^N}^g)}{\mathrm{Area}_g(\Sigma)}
\end{equation}
are elements of the line $|\det|_\Sigma^{\otimes-2}$. If $\Sigma$ has multiple connected components, and if we assign different boundary conditions on different components of the boundary, the same holds where $\mathrm{Area}_g(\Sigma)$ is replaced with the product of the areas of connected components of $\Si$ that have no Dirichlet boundary. Alternatively, one can restrict these functionals to normalized metrics, and drop the area correction.

\subsubsection{The determinant line associated to a loop}

If $\Sigma$ is a Riemann surface, recall that we denote by ${\mathcal L}(\Sigma)$ the set of simple loops drawn on $\Sigma$. For any loop $\ell\in{\mathcal L}(\Sigma)$, we define
$$
|\det|_{\ell,\Sigma}=|\det|_\Si\otimes|\det|_{\Si\setminus\ell}^{-1}.
$$

The assignment $\ell\mapsto|\det|_{\Sigma,\ell}$ defines an oriented line bundle $|\Det|_\Sigma$ over the base space ${\mc L}(\Sigma)$, the \emph{determinant line bundle}.

\begin{prop}\label{prop:incldetbund}
Any embedding $\Si'\hookrightarrow \Si$ induces a map $\phi_{\Si'}^{\Si}:|\Det|_{\Si'}\hookrightarrow|\Det|_\Si$ between the associated determinant line bundles.

This map is natural in the sense that given Riemann surfaces $\Si_1\hookrightarrow \Si_2 \hookrightarrow \Si_3$, we have $\phi_{\Si_1}^{\Si_2}\circ\phi_{\Si_2}^{\Si_3}=\phi_{\Si_1}^{\Si_3}$.
\end{prop}

Let us first note that any embedding $\Si'\hookrightarrow \Si$ induces an embedding ${\mc L}(\Sigma')\hookrightarrow{\mc L}(\Sigma)$ of the spaces of simple loops, i.e. of the base spaces of the determinant line bundles.
Given two Riemann surfaces $\Si'\subset \Si$ and a simple loop $\ell\in{\mc L}(\Si')$, we will define $\phi_{\Si'}^{\Si}$ by giving its restriction $\phi_{\Si'|\l}^{\Si}$ to fibers, which is a natural isomorphism between the determinant lines
$$
|\det|_{\ell,\Sigma'}\simeq|\det|_{\ell,\Sigma},
$$
i.e. a natural isomorphism
$$
|\det|_\Si'\otimes|\det|_{\Si'\setminus\ell}^{-1}\simeq|\det|_{\Si}\otimes|\det|_{\Si\setminus\ell}^{-1}.
$$

\begin{defn}\label{def:neutcollmet}
A \emph{neutral collection} of metrics $(g,g_\l,g',g'_\l)$ is the data of four well-behaved metrics on $\Si$, $\Si\setminus\ell$, $\Si'$ and $\Si'\setminus\ell$ that satisfy the following property.
There exists a set $A$, disjoint union of two annuli, one in each component of a tubular neighborhood of $\l$ in the surface $\Si'\setminus\l$ (so that $A$ disconnects a small neighborhood of the loop $\ell$ from points of $\Si$ and $\Si'$ that are away from the loop $\ell$; inside of $A$ means near $\ell$ and outside of $A$ means away from $\l$) such that all four metrics agree on $A$; $g$ and $g_\l$ (resp. $g'$ and $g'_\l$) agree outside of $A$; $g$ and $g'$ (resp. $g_\l$ and $g'_\l$) agree inside of $A$.
\end{defn}

Given a neutral collection of metrics $(g,g_\l,g',g'_\l)$ on $\Si$, $\Si\setminus\ell$, $\Si'$ and $\Si'\setminus\ell$, we would like to define the isomorphism $\phi_{\Si'|\l}^{\Si}$ by
\begin{equation}\label{eq:isomdet}
[g]\otimes[g_\l]^{-1}\simeq [g']\otimes [g'_\l]^{-1}.
\end{equation}
Given another choice of a neutral collection $(e^{2\sigma}g,e^{2\sigma_\l}g_\l,e^{2\sigma'}g',e^{2\sigma'_\l}g'_\l)$, we can assume without loss of generality (by cutting and pasting) that the union of annuli $A$ in Definition \ref{def:neutcollmet} is the same for our two neutral collections. In particular, $\sigma=\sigma'=\sigma_\l=\sigma'_\l$ on the set $A$; $\sigma$ and $\sigma_\l$ (resp. $\sigma'$ and $\sigma'_\l$) agree outside of $A$; $\sigma$ and $\sigma'$ (resp. $\sigma_\l$ and $\sigma'_\l$) agree inside of $A$. By locality of the Liouville action, we then have
$$
S_L(g,\sigma) - S_L(g_\l,\sigma_\l) = S_L(g',\sigma') - S_L(g'_\l,\sigma'_\l),
$$
so that the isomorphism $\phi_{\Si'|\l}^{\Si}$ - as defined in (\ref{eq:isomdet}) - does not depend on the choice of a neutral collection of metrics. 

We can rephrase the fact that (\ref{eq:isomdet}) is a non-ambiguous definition in the following way (where $\det_\z$ denotes either $\det_\z$ or $\det'_\z$, and boundary conditions are Neumann on $\l$ and Dirichlet elsewhere):

\begin{prop}\label{prop-defcocycle}
The quantity
$$
\M(\Si,\Si';\l) = \log \left(\frac{\det_\z(\D^{g'}_{\Si'})\det_\z(\D^{g_\l}_{\Si\setminus\l})}{\det_\z(\D^{g}_{\Si})\det_\z(\D^{g'_\l}_{\Si'\setminus\l})}\right)
$$
is independent of the choice of a neutral collection of normalized metrics $(g,g_\l,g',g'_\l)$.
\end{prop}

From the definition of $\M$, we get the following cocycle property.

\begin{prop}[Cocycle property]\label{prop:cocycle}Suppose we have three Riemann surfaces $\Si_1\subset \Si_2 \subset \Si_3$, and a loop $\l \subset \Si_1$. Then $\M(\Si_3,\Si_2;\l) + \M(\Si_2,\Si_1;\l) = \M(\Si_3,\Si_1;\l)$.
\end{prop}

Note that, from the Polyakov-Alvarez anomaly formula, other choices of boundary conditions in Definition \ref{prop-defcocycle} would produce a cocycle $\widetilde{\M}$ differing from $\M$ by a coboundary $f$: $\widetilde{\M}(\Si,\Si';\l) = \M(\Si,\Si';\l) + f(\Si;\l) - f(\Si';\l)$.

We can now finish the proof of Proposition \ref{prop:incldetbund}.

\begin{proof}
The cocycle property \ref{prop:cocycle} shows that the composition of canonical isomorphisms $\phi$ between the lines $|\det|_{\ell,\Sigma_i}$ is itself a canonical isomorphism. As a consequence, the line $|\det|_{\ell,\Sigma}$ depends only on the loop $\ell$ and on an arbitrarily thin tubular neighborhood of it. In particular, any embedding $\xi:\Sigma'\hookrightarrow\Sigma$ extends to the bundle map $|\Det|_{\Sigma'}\hookrightarrow|\Det|_{\Sigma}$ in a way consistent with composition of maps.
\end{proof}

We will later see that the quantity $e^{-\M(A, A';\l)}$ correspond to the Radon-Nikodym derivative of our continuous $\SLE_2$ loop measure under restriction to a smaller annulus $A'\subset A$ (Proposition \ref{prop:covar-conf}).

\begin{rem}\label{rem:MregBl}
Comparing the definition of $\M$ to Lemma \ref{lem:ratio-det}, we can think of $\M$ as a regularization of a certain mass of Brownian loops, namely $\mu^{\lo}_{\Si}\{l|l\cap (\Si \sm \Si')\neq\emptyset\} - \mu^{\lo}_{\Si \sm \l}\{l|l\cap (\Si \sm \Si')\neq\emptyset\}$.
\end{rem}

\begin{rem}\label{rem:cont-cocycle}
The $\zeta$-regularized determinants of a smooth family of Laplacians is a smooth function (\cite{BerGetVer_HeatKernels}, Proposition 9.38). Hence the quantity $\M(\Si,\Si';\l)$ is regular in its parameters. In particular, it is measurable in $\l$ (for the Borel $\sigma$-algebra associated to the topology $\Top$).
\end{rem}

\subsection{Uniform spanning trees}

\subsubsection{Definitions}

Let $\G$ be a connected graph, possibly with a boundary (i.e. a distinguished subset of vertices). Fix some boundary conditions on $\G$, by declaring some of the boundary vertices wired and the others free. We can build from $\G$ a graph $\widetilde{\G}$ that encodes these boundary conditions, by contracting all the wired boundary vertices of $\G$ into one distinguished vertex and by deleting all free vertices, as well as edges having a free vertex as one of their endpoints.

\begin{defn}\label{def:UST}
The uniform spanning tree (UST) on $\G$ is the uniform measure on the set of subgraphs of $\widetilde{\G}$ that contain all vertices of $\widetilde{\G}$ (i.e. are spanning), and that are connected and cycle-free (i.e. are trees).
It is seen as a measure on subgraphs of $\G$.
\end{defn}

Suppose now that $\G$ is a connected and finite subgraph of $\Z^2$, and consider one of its spanning trees $T$. The graph $T^\dagger$ dual to $T$ is the set of edges of the grid $(\frac{1}{2},\frac{1}{2})+\Z^2$ that intersect edges of $\G\sm T$. 

\begin{defn}\label{def:explproc}
The exploration process $e$ of the spanning tree $T$ is a path drawn on the graph $(\frac{1}{4},\frac{1}{4})+\frac{1}{2}\Z^2$. It consists of all edges neighboring $\G$, that do not intersect $T\cup T^\dagger$.
\end{defn}

The curve $e$ is the interface between the spanning tree $T$ and its dual graph $T^\dagger$, it is a simple curve that follows the contour of the tree $T$ as closely as possible. One can actually build it as an exploration of the spanning tree $T$: take a boundary point (in $(\frac{1}{4},\frac{1}{4})+\frac{1}{2}\Z^2$) of the planar graph $\G$, and move forward in $\G$, by choosing the rightmost path in $(\frac{1}{4},\frac{1}{4})+\frac{1}{2}\Z^2$ that does not cross any edge of $T$. This will draw some connected component of the exploration process of $T$.

If we know the first steps of the exploration process $e$, we get local information on the tree $T$: the edges of $\G$ sitting to the right of $e$ are in $T$, whereas dual edges to its left are in the dual tree $T^\dagger$.
\begin{figure}[htb]
\begin{center}
\includegraphics[width = 7cm]{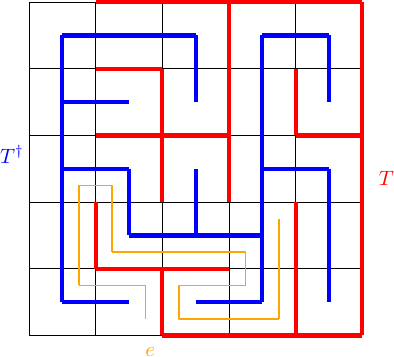}
\caption{A spanning tree $T$ with mixed boundary conditions, its dual tree, and the first steps of its exploration process $e$.}
\end{center}
\end{figure}

The UST measure has a spatial Markov property.

\begin{prop}\label{prop:USTMarkov}
Conditioned on the initial steps of the exploration process $e$, the law of $T$ in the unexplored domain has the law of a UST with free boundary conditions on the left side of $e$, and wired boundary condition on its right side.
\end{prop}

Wilson gave in \cite{Wi} an algorithm to sample from the UST measure, by generating branches as loop-erased random walks. Boundary conditions are enforced by having the random walks be reflected (resp. stopped) upon hitting the free (resp. wired) boundary. This relates USTs to simple random walks, and hence to discrete harmonic analysis. In particular, wired and Dirichlet boundary conditions should correspond to each other, and similarly for free and Neumann boundary conditions. 

This connection between USTs and harmonic analysis will allow us to rewrite the Radon-Nikodym derivative of exploration processes. Let $\#T(\G)$ be the number of spanning trees on a graph $\G$. It is related to the discrete Laplacian on $\G$.

\begin{thm}[Matrix-tree theorem]
Suppose the wired boundary of $\G$ is non-empty.
Then $\#T(\G) = \det(\D_{\G})$.
\end{thm}

We refer the interested reader to e.g. Theorem 1.19 in \cite{HarHirMos_GraphTheory} for the proof of another version of this theorem.

\subsubsection{Radon-Nikodym derivatives of USTs}\label{subsec:spandet}

Let us get back to our annular set-up: consider a discrete annulus $A^\d$, a cut $c^\d$, and two cuts $d_1^\d$ and $d_2^\d$ that disconnect $c^\d$ from points of $A^\d$ that are at distance more than $\e$ from $c^\d$.

\begin{lem}\label{lem:DRN-expl}
The exploration process $e^\d_{A^\d}$ is absolutely continuous with respect to the process $e^\d_{A^\d\sm c^\d}$, until the first hitting time $T_{c^\d}^{\e}$ of the $\e$-neighborhood of the cut $c^\d$.

Explicitly, letting $\g^\d$ be a curve defined until it first hits the $\e$-neighborhood of the cut $c^\d$, we have (for any mesh size $\d \ll \e$):
$$
\frac{\P[e^\d_{A^\d}\mathrm{\ starts\ by\ }\g^\d]}{\P[e^\d_{A^\d\sm c^\d}\mathrm{\ starts\ by\ }\g^\d]}  = \frac{\det(\Id - \Ha^{c^\d\rightarrow d^\d}_{A^\d\sm{K^\d}} \circ \Ha^{d^\d\rightarrow c^\d}_{A^\d\sm{K^\d}})} {\det(\Id - \Ha^{c^\d\rightarrow d^\d}_{A^\d} \circ \Ha^{d^\d\rightarrow c^\d}_{A^\d})},
$$
where $K^\d$ is the image of $\g^\d$, and carries a natural Neumann arc $\B_{K^\d} = \overline{a^\d,\g^\d(T_{c^\d}^{\e})}$.
\end{lem}

\begin{proof}
The Radon-Nikodym derivative is a ratio of numbers of trees, that can be rewritten thanks to the matrix-tree theorem:
$$
\frac{\P[e^\d_{A^\d}\mathrm{\ starts\ by\ }\g^\d]}{\P[e^\d_{A^\d\sm c^\d}\mathrm{\ starts\ by\ }\g^\d]}  = \frac{\frac{\#T(A^\d\sm K^\d)}{\#T(A^\d)}}{\frac{\#T(A^\d\sm (K^\d\cup c^\d))}{\#T(A^\d\sm c^\d)}} = \frac{\det(\D_{A^\d\sm K^\d})\det(\D_{A^\d\sm c^\d})}{\det(\D_{A^\d})\det(\D_{A^\d\sm (K^\d\cup c^\d)})}.
$$
Cutting along $d_1^\d$ and $d_2^\d$ disconnects $c^\d$ from $K^\d$ in the domain $A^\d$. Hence, we have
\begin{equation}\label{eq:1}
1 = \frac{\det(\D_{A^\d\sm (K^\d\cup c^\d\cup d^\d)})\det(\D_{A^\d\sm d^\d})}{\det(\D_{A^\d\sm (c^\d\cup d^\d)})\det(\D_{A^\d\sm (K^\d\cup d^\d)})}.
\end{equation}
We can thus multiply the above Radon-Nikodym derivative by (\ref{eq:1}), to get an expression involving $8$ determinants of Laplacians. Using Lemma \ref{lem:ratio-det} twice yields the claim:
$$
\frac{\P[e^\d_{A^\d}\mathrm{\ starts\ by\ }\g^\d]}{\P[e^\d_{A^\d\sm c^\d}\mathrm{\ starts\ by\ }\g^\d]}  = \frac{\det(\Id - \Ha^{c^\d\rightarrow d^\d}_{A^\d\sm{K^\d}} \circ \Ha^{d^\d\rightarrow c^\d}_{A^\d\sm{K^\d}})} {\det(\Id - \Ha^{c^\d\rightarrow d^\d}_{A^\d} \circ \Ha^{d^\d\rightarrow c^\d}_{A^\d})}.
$$
\end{proof}
\subsection{A brief word on SLE}

Chordal Schramm-Loewner evolutions (SLEs) form a one parameter family of conformally invariant random curves defined in simply-connected domains of the complex plane, with prescribed starting point and endpoint on the boundary. They are not simple curves, but will not cross their past paths (when touching their past, they will bounce off it).

Let us first give the definition of $\SLE_\kappa$ in the upper half-plane $(\mathbb{H},0,\infty)$. It is a random curve $\g : \R^+ \rightarrow \overline{\mathbb{H}}$, growing from the boundary point $0$ to $\infty$.

Suppose that such a curve $\g$ is given to us. Let $H_s$ be the unbounded connected component of $\mathbb{H} \sm \g([0,s])$, and consider the uniformizing map $g_s : H_s \rightarrow \mathbb{H}$, normalized at $\infty$ such that $g_s(z) = z + 2a_s/z + o(1/z)$. The quantity $a_s$ is the so-called half-plane capacity of the compact hull $K_s=\H\sm H_s$ generated by $\g([0,s])$. Under additional assumptions\footnote{The curve $\gamma$ needs to be instantaneously reflected off its past and the boundary in the following sense: the set of times $s$ larger than some time $s_0$ that $\gamma$ spends outside of the domain $H_{s_0}$ should be of empty interior.}, the half-plane capacity $a_s$ is an increasing bijection of $\R^+$, and so we can reparametrize our curve by $t=a_s$.

With this parametrization, the family of functions $g_t$ solves the Loewner differential equation:
\begin{equation} \nonumber
\left\{\begin{array}{l}g_0(z) = z\\
\partial_t g_t(z) = \frac{2}{g_t(z)-W_{t}},\end{array}\right.
\end{equation}
where $W_t= g_t(\g_t)$ is the (real-valued) driving function.

Conversely, starting from a continuous real-valued driving function, it is always possible to solve the Loewner equation, and hence to recover a family of compact sets $K_t$ in $\overline{\H}$, growing from $0$ to $\infty$, namely $K_t$ is the set of initial conditions $z$ that yield a solution $g_u(z$) blowing up before time $t$. It may happen that the compact sets $K_t$ coincides with the set of hulls generated by the trace of a curve $\g$, which can in this case be recovered as $\g_t= \lim_{\e\rightarrow 0} g_t^{-1}(W_t +i \e)$.

\begin{prop}\label{prop:SLE}
The process $\SLE_\kappa^\H(0\rightarrow\infty)$ is the curve obtained from the solution of the Loewner equation with driving function $W_t = \sqrt{\kappa} B_t$, where $B_t$ is a standard Brownian motion.
\end{prop}

The law of $\SLE_\kappa^\H(0\rightarrow\infty)$ is invariant by scaling. Hence, given a simply-connected domain $(\O,a,b)$ with two marked points on its boundary, we can define $\SLE_\kappa^\O(a\rightarrow b)$ to be the image of an $\SLE_\kappa^\H(0\rightarrow\infty)$ by any conformal bijection $(\H,0,\infty) \rightarrow (\O,a,b)$.

SLE curves have a spatial Markov property built into them similar to the one satisfied by the exploration process of the UST (Proposition \ref{prop:USTMarkov}):

\begin{prop}\label{prop:markov-sle} The law of $\SLE_\kappa^\H(0\rightarrow\infty)$ after a stopping time $\tau$ conditioned on its past has the law of an $\SLE_\kappa^{H_\tau}(\g_\tau\rightarrow\infty)$.
\end{prop}

Certain SLEs arise as the scaling limits of discrete curves coming from statistical mechanics. In particular, $\SLE_8$ appears in the scaling limit of USTs. Let $(\O,x_0,a,b)$ be a simply-connected domain of the plane, with two points $a,b$ marked on its boundary. Let us consider its natural approximation $\O^\d$ at mesh size $\d$, and look at the UST $T^\d$ on $\O^\d$, with wired boundary conditions on the counterclockwise arc $\arcperso{a^{\d}b^{\d}}$, and free boundary conditions on $\arcperso{b^{\d}a^{\d}}$. Let $e^\d$ be the exploration process of $T^\d$ going from $a^{\d}$ to $b^{\d}$.

\begin{thm}[\cite{LSW_LERW}, Theorem 4.4]\label{prop:USTcvSLE}
The discrete curve $e^\d$ converges\footnote{In the sense of convergence of Loewner driving functions (after uniformization to $\H$), for the topology of uniform convergence on compact sets of times. Convergence of the curve itself follows when the boundary of $\O$ is smooth enough, e.g. piecewise $\C^1$.} towards $\SLE_8^\O(a\rightarrow b)$, when the mesh size $\d$ goes to $0$, uniformly on Carath\'eodory-compact sets of decorated domains $(\O,x_0,a,b)$.
\end{thm}

Some geometric properties of SLE$_8$ can be easily deduced from this convergence, in particular that it has to be a reversible space-filling curve.

We will need a loop version of $\SLE_8$, where the starting and end points are the same.

\begin{prop}\label{prop:bubbleSLE}
Consider a Jordan domain $(\O,a,b)$, and let the counterclockwise arc $\arcperso{ba}$ shrink to $a$. $\SLE_8^\O(a\rightarrow b)$ then converges (weakly for the topology of uniform convergence up to reparametrization) towards a random counterclockwise loop, that we call counterclockwise $\SLE_8(2)$.
\end{prop}

\begin{proof}
Let us fix a spectator point $o \in \partial \O$ distinct from $a$, as well as a uniformizing map $\phi_o : (\O,a,o) \rightarrow (\H,0,\infty)$. Stop $\SLE_8^\O(a\rightarrow b)$ at the first time $\tau_o$ when its trace disconnects $o$ from $b$, and push it by $\phi_o$ to obtain a random curve $\gamma$ growing in the upper half-plane. The driving function $W_t$ of $\gamma$ can be seen to be an explicit functional $F_t(X)$ of a Bessel process $X$ of dimension $2$ started from $X_0=\phi_o(b)$, namely $F_t(X) = \sqrt{8} X_t - \int_0^t \frac{1}{\sqrt{2} X_s} ds$ (see e.g. \cite{SchWi_SLECoordChange}).

Now, when the arc $\arcperso{ba}$ shrinks to the point $a$, the process $X$ converges to a 2-dimensional Bessel process started from $0$. The quantity $\int_0^t \frac{1}{X_s} ds$ being a.s. finite (even when $X_0 = 0^+$), the driving function $W_t=F_t(X)$ also converges. This implies convergence of the curve $\phi_0\left(\SLE_8^\O(a\rightarrow b)\right)$ up to time $\tau_o$, which in turn gives convergence of the curve $\SLE_8^\O(a\rightarrow b)$ up to time $\tau_o$. If we now let the counterclockwise arc $\arcperso{oa}$ shrink to $a$, we get convergence of the whole curve $\SLE_8^\O(a\rightarrow b)$.
\end{proof}

The random curve $\SLE_8(2)$ also fits into a two-parameter family of solutions to Loewner SDEs (see e.g. \cite{LSW3}), and the notation refers to the values of the parameters $\kappa$ and $\rho$.

\section{Tightness}\label{precomp}

The goal of this section is to establish the tightness of the UST exploration process. Wilson proved in \cite{Wi} that the branches of the UST can be constructed as loop-erased random walks. As a consequence, one can use simple random walk estimates - or equivalently discrete harmonic analysis - to get a priori estimates on the UST and its exploration process, which in particular imply that the exploration process of the UST (in the bulk or close to a piecewise-$\mathcal{C}^1$ boundary) form tight families.

We are going to follow very closely Schramm's argument in \cite{S0}, where he considered the simply-connected setting. There, Schramm used the fact that the graph dual to the uniform spanning tree is itself a uniform spanning tree, and hence can also be generated via Wilson's algorithm. However, in a non-simply-connected domain, graphs dual to a spanning tree are not trees, and cannot be generated by Wilson's algorithm (it is however possible to modify the original algorithm to generate these dual graphs, see \cite{KasKen_RandomCurves}). Consequently, some of the proofs in \cite{S0} will not exactly work as is. However, the use of stochastic comparison will easily allow us to transfer these estimates - or at least their proofs - to our setting.

First of all, let us state the stochastic comparison lemma, which is a consequence of negative correlation for the UST on a general graph $\G$.

\begin{lem}[Stochastic comparison]\label{lem:stocdom}
Let $\mathcal{I}$ be a collection of edges of $\G$, and let $\mathcal{I}_1 \subset \mathcal{I}_2$ be two subsets of $\mathcal{I}$ that can be completed in spanning trees of $\G$ by adding edges of $\mathcal{I}^c$.

Let $T_1$ (resp. $T_2$), be the uniform spanning tree $T$ on $\G$, conditioned on $T \cap \mathcal{I} = \mathcal{I}_1$ (resp. on $T \cap \mathcal{I} = \mathcal{I}_2$).

Then, there exists a coupling of $T_1$ and $T_2$ such that $T_2\cap\mathcal{I}^c \subset T_1\cap\mathcal{I}^c$ almost surely.
\end{lem}

\begin{proof}
There is a discussion of this fact following Remark 5.7 in \cite{BenLyoPerSch_USF}.
For a nice overview of the Strassen domination theorem, we refer to \cite{Lin_StochDom}.
\end{proof}

Let us now investigate the geometry of the wired UST on discrete approximations of a conformal annulus, and in particular the occurrence of certain $n$-arm events\footnote{An $n$-arm event at a point $x$ is the existence of $n$ disjoint branches of the tree - the so-called arms - connecting an infinitesimal neighborhood of $x$ to some points at positive distance. The different arms may or may not be connected to each other at $x$.}.

\begin{lem}[4-arm estimate]\label{lem:4-arm}
For all $\e > 0$, there is a radius $r\in (0,\e)$ such that for all small enough mesh size $\d$, the following holds: the probability that there is a point $p$ $\e$-inside the domain such that there are $4$ disjoint branches of the uniform spanning tree crossing the circular annulus $A(p,r,\e)$ is less than $\e$.
\end{lem}

\begin{proof}
The equivalent statement in the simply-connected case is proved as Corollary 10.11 in \cite{S0}.

We can use stochastic comparison to transfer Schramm's 4-arm estimate in simply-connected domains to an annulus $A^\d$. Indeed, it is possible to cover $A^\d$ by finitely many simply-connected domains $B_i^\d$ such that any ball of radius $\e$ in $A^\d$ is included in one of the $B_i^\d$. The trace on a subdomain $B_i^\d$ of the UST in $A^\d$ is itself a UST, with certain random boundary conditions given by the connections in the UST outside of $B_i^\d$. Whatever they are, these random boundary conditions are always more wired than free boundary conditions. Hence the trace on $B_i^\d$ of the UST in $A^\d$ has no more edges (Lemma \ref{lem:stocdom}) than the UST in $B_i^\d$ with free boundary conditions.

In other words, any 4-arm event in the annulus $A^\d$ gives a 4-arm event in one of the simply-connected domains $B_i^\d$, and this is very unlikely.
\end{proof}

We just proved that no 4-arm events happen in the scaling limit. Can we say something on 3-arm events ? 3-arm events around a point $x$ are of three different kinds. One possibility is that all three arms are connected at the point $x$. This corresponds to a branching point of the spanning tree, and such events happen almost surely (and densely). Another possibility (dual to the first one) is that no two arms are connected at the central point: this corresponds to a branching point of the dual graph, and, similarly, there are many such points. The last possible configuration would be a path escaping from the neighborhood of a branch: 2 of the arms are connected together at the point $x$, the last one is not. This last kind of 3-arm does not exist in the scaling limit.
\begin{figure}[htb]
\begin{center}
\includegraphics[width = 12cm]{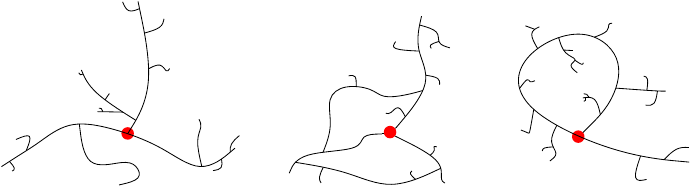}
\caption{The different kinds of 3-arm events.}
\end{center}
\end{figure}

\begin{lem}\label{lem:pointsclosetoabranchareconnectedtoit}
For all $\e > 0$, there is a radius $r\in (0,\e)$ such that for all small enough mesh size $\d$, with probability larger than $1-\e$, if there are three disjoint branches of the uniform spanning tree crossing any circular annulus $A(p,r,\e)$ where $p$ is $\e$-inside the domain, either the three branches are all connected to each other in the ball of radius $r$ around $p$, or they are all connected to each other outside of the ball of radius $\e$ around $p$.
\end{lem}

\begin{proof}
We refer to the proof of Theorem 10.7 in \cite{S0}. We cannot directly use stochastic domination to transfer the result, as the event we are trying to control is preserved neither by adding nor by erasing edges. But we can readily transfer its proof, using our 4-arm estimate in the doubly-connected setting (Lemma \ref{lem:4-arm}) whenever Schramm calls for his Corollary 10.6 or Lemma 10.8.
\end{proof}

Close to a piecewise $\mathcal{C}^1$ boundary (assuming full wired boundary conditions), we have similar arm estimates.

\begin{lem}[Boundary 2-arm estimate]\label{lem:boundary-2-arm}
For all $\e > 0$, there is a radius $r\in (0,\e)$ such that for all small enough mesh size $\d$, the following holds: the probability that there exists a boundary point $p$ such that there are $2$ disjoint branches of the uniform spanning tree crossing the circular annulus $A(p,r,\e)$ is less than $\e$.
\end{lem}

\begin{proof}
In the simply-connected case, this is a consequence of \cite{S0}, Theorem 11.1, (i) and (ii). We can use stochastic domination to transfer this result to an annular domain.
\end{proof}

\begin{lem}\label{lem:dualtrunkdisjointfromboundary}
For all $\e > 0$, there is a radius $r\in (0,\e)$ such that for all small enough mesh size $\d$, with probability at least $1-\e$, all points $r$-close to the boundary are connected to it in the tree within a ball of radius $\e$.
\end{lem}

\begin{proof}
This is a companion statement to Lemma \ref{lem:pointsclosetoabranchareconnectedtoit} on the boundary, and the proof is similar.
\end{proof}

The 4-arm and boundary 2-arm estimates (Lemmas \ref{lem:4-arm} and \ref{lem:boundary-2-arm}) are enough to give tightness on the exploration process of the UST.

\begin{prop}\label{prop:precompexpl}
The exploration processes $e^\d_{A^\d}$ and $e^\d_{A^\d\sm c^\d}$ in natural approximations of $A$ and $A\sm c$ are tight for the topology $\Top$.
\end{prop}

\begin{proof}
To any integer sequence $(N_i)_{i\in\N}$, we can associate a subset of paths in $A$ that is compact for the topology $\Top$ of uniform convergence up to reparametrization: the set of paths $\g$ such that for any integer $i$, $\g$ makes fewer than $N_i$ steps\footnote{Any reasonable definition would work here. For example, one can count steps of size $\epsilon$ inductively: given $p_n$ the endpoint of the $n$-th $\epsilon$-step, one can let the $(n+1)$-th step end when the path $\gamma$ exits the ball of radius $\epsilon$ around $p_n$.} of size $2^{-i}$. Indeed, these sets of paths satisfy the Bolzano-Weierstrass property, as one can iteratively extract subsequences so that the number of steps of size $2^{-i}$, as well as their endpoints, converge.

Let us now show that, independently of the mesh size $\d$, with an arbitrarily high probability, the exploration process is in one of these compact sets of paths. From Lemmas \ref{lem:4-arm} and \ref{lem:boundary-2-arm} we know that for any $\e$, we can find some small $r(\e)$ such that with probability bigger than $1-\e$, there is no circular annulus $A(p,r(\e),\e)$ crossed by more than three disjoint branches of the UST, and so the exploration process crosses no annulus of radius $A(p, r(\e), \e)$ more than $3\times2=6$ times\footnote{Note that a crossing of the annulus by the exploration process gives a crossing by the spanning tree on its right-hand side.}.
Hence, with probability bigger than $1-\e$, the curve makes a number of steps of size $\e$ that is bounded by $6$ times the number $N(r(\e))$ of balls of radius $r(\e)$ needed to cover the annulus $A$.
Tightness follows.
\end{proof}

The arm estimates give information on the regularity of any limiting curve of the exploration process. Let us call $e$ a subsequential limit of the UST in some domain $D$, and let $\tau$ be a stopping time of $e$. Notice that for topological reasons\footnote{The same statement would hold if we replaced $e([0,\tau])$ by any closed subset of $\overline{D}$.}, the connected components of $D \sm e([0,\tau])$ are open.
\begin{figure}[htb]
\begin{center}
\includegraphics[width = 5cm]{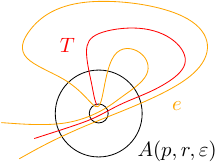}
\caption{A bad event for the exploration process $e$, and the 3-arm events it forces.}
\end{center}
\end{figure}

\begin{prop}\label{prop:nohidingloops}
Almost surely, there is a connected component $\O$ of $D \sm e([0,\tau])$ carrying $e(\tau)$ on its boundary such that $(e_t)_{t\geq\tau}$ stays in $\overline{\O}$.
Moreover, the sets of time when $e$ is inside $\O$ contains $\tau$ in its closure.
\end{prop}

\begin{proof}
Any contradiction of the preceding statement would allow us to see, on the discrete tree, the type of 3-arm events which is prevented by Lemmas \ref{lem:pointsclosetoabranchareconnectedtoit} and \ref{lem:dualtrunkdisjointfromboundary}.
\end{proof}

Note that we can apply Proposition \ref{prop:nohidingloops} at the countable family of stopping times $\tau_N^n$ corresponding to the endpoint of the $N$-th $2^{-n}$-step. This yields an almost sure control on the entire behavior of the curve $e$ at once.

\section{The exploration process of the UST in an annulus}\label{expl}

\subsection{Convergence of the reference process $e^\d_{A^\d\sm c^\d}$ to $\SLE_8(2)$}\label{cv-SLE8}

We are now  going to discuss the convergence of the exploration processes of a wired UST on the natural approximations of a simply-connected piecewise $\mathcal{C}^1$ domain $D$ (e.g. $A\sm c$) towards $\SLE_8(2)$, which is a minor degeneration of a well-known result by Lawler, Schramm and Werner (Theorem \ref{prop:USTcvSLE}). The idea of the proof is straightforward. We have tightness of the laws under consideration, hence we only need to characterize uniquely any subsequential limit in order to show convergence. To do so, we let the exploration process evolve for a very short amount of time in order to produce free boundary conditions in the domain yet to be explored. The continuous exploration process in this remaining domain is seen to be an $\SLE_8$ (using the convergence result of \cite{LSW_LERW}). This is enough to characterize as $\SLE_8(2)$ any subsequential limit of our exploration processes.

\begin{thm}\label{thm:cv-explbubble}
The exploration process $e^\d_{D^\d}$ converges in law (for the topology $\Top$) towards counterclockwise $\SLE_8(2)$.
\end{thm}

\begin{proof}The exploration processes of the UST in natural approximations of a domain $(D,x_0)$ starting from a boundary point $a$ form a tight family for the topology $\Top$ (Proposition \ref{prop:precompexpl}), so we can consider some subsequential limit $e_{D}$ (in the almost sure sense, thanks to Skorokhod's representation theorem).

Let $\tau_{\e}^\d$ be the exit time of the ball of radius $\e$ centered at $a$, that we may assume (up to further extracting) to converge to a time $\widetilde{\tau_\e}$ (that may \emph{a priori} be larger than the first exit time $\tau_\e$).

Call $D_\e$ the connected component of $D\sm e_D([0,\widetilde{\tau_\e}])$ containing $a$ on its boundary, which is a simply-connected domain with distinct boundary points $a$ and $e_D(\widetilde{\tau_\e})$ (follows from Proposition \ref{prop:nohidingloops}). Notice that, in particular, the domain $D_\e$ carry non-trivial free boundary conditions. Let us now consider the evolution of the exploration process after time $\widetilde{\tau_\e}$. The subsequential limit $\g_\e$ of $(e^\d_{D^\d}(t))_{t\geq\tau_{\e}^\d}$ has the law of an $\SLE_8^{D_\e}(e_D(\widetilde{\tau_\e})\rightarrow a)$ (see Lemma \ref{lem:SLEepsilon}).

Let us fix an observation point $o$ on the boundary $\partial D$, and let $\tau_{o}$ be the first hitting time of $o$ by $\g_\e$. We push our curve to the upper half-plane $\H$ using the conformal map $\phi : (D,a,o) \rightarrow (\H,0,\infty)$, that moreover satisfies $|\phi(x_0)|=1$. By SLE change of coordinates (see e.g. \cite{SchWi_SLECoordChange}), the curve $\phi(e_D)$ from the time $\widetilde{\tau_\e}$ until time $\tau_{o}$ is driven by a certain functional $c_\e + F_t(X^{a_\e})$ of a Bessel process $X^{a_\e}$ of dimension $2$ started from $X^{a_\e}_0=a_\e$, where $F_t(X) = \sqrt{8} X_t - \int_0^t \frac{1}{\sqrt{2} X_s} ds$. Moreover, the random variables $c_\e$ and $a_\e$ can be bounded by a quantity going to $0$ as $\e$ goes to $0$. Indeed, $a_\e$ and $c_\e$ can be controled by the size of the ball $B_\e(a)$ in the domain $D$ seen from the point $o$, where we measure size using the image of the Lebesgue measure (seen as a measure on $\partial \H \sm \{\infty\}$) by a fixed uniformizing map. For example, the quantity $a_\e$ is by definition exactly the size of the arc $\arcperso{a^\d,\g^\d_\e(0)}$ in the domain $D_\e$. Hence, $a_\e$ is less than the size of the boundary of the ball $B_\e(a)$ in the domain $D \sm B_\e(a)$, as conformal images of the Lebesgue measure enjoy a monotonicity property (conformal images of the Lebesgue measure can also be seen as the exit measure of Brownian excursions starting at $o$). This gives a bound on $a_\e$ that does not depend on the exploration process.

Hence, the process $e_D$ until time $\tau_o$ is driven by $F_t(X^{0^+})$, and so has the law of an $\SLE_8(2)$.
\end{proof}

\begin{lem}\label{lem:SLEepsilon}
For any subsequence of mesh sizes for which the process $e^\d_{D^\d}$ converges, conditionally on $(D_\e,e_D(\widetilde{\tau_\e}))$, the limit $\g_\e$ of $(e^\d_{D^\d}(t))_{t\geq\tau_{\e}^\d}$ has the law of an $\SLE_8^{D_\e}(e_D(\widetilde{\tau_\e})\rightarrow a)$.
\end{lem}

\begin{proof}
Let us denote by $\g^\d_\e$ the exploration process $e^\d_{D^\d}$ after its exit time $\tau_{\e}^\d$ of the ball of radius $\e$ centered at $a$.
Thanks to the Markovian property of the UST (Proposition \ref{prop:USTMarkov}), we know that  $\g^\d_\e$ has the law of the exploration process of an independent UST in the remaining domain $D_\e^\d$, with wired boundary conditions on the counterclockwise boundary arc $\arcperso{\g^\d_\e(0),a^\d}$ and free boundary conditions on the arc $\arcperso{a^\d,\g^\d_\e(0)}$.

Hence the Loewner driving function $W_\e^\d$ of the curve $\g^\d_\e$ converges in law (for the topology of locally uniform convergence) towards the driving function $W_t={B_{8t}}$ of an $\SLE_8$ (Theorem \ref{prop:USTcvSLE}). To conclude, we need to show that any subsequential limit $\g_\e$ for the topology $\Top$ is driven by $W$.

Consider the uniformizing map $\phi_\e : (D_\e,e_D(\widetilde{\tau_\e}),a) \rightarrow (\H,0,\infty)$ that moreover satisfies $|\phi_\e(x_0)|=1$. We will call capacity the half-plane capacity of objects in $D_\e$ pushed by the map $\phi_\e$.

The curve $\g_\e$ can be parametrized by capacity. Indeed, if capacity were not to give a parametrization of $\g_\e$, we could find a time interval $[t_1,t_2]$ on which $\g_\e$ does not grow in capacity. Which is equivalent to say that, during the time interval $[t_1,t_2]$, $\g_\e$ would have to stay away from the interior of the connected component of $D_\e\sm \g_\e([0,t_1])$ containing $a$ on its boundary. This would contradict Proposition \ref{prop:nohidingloops}.

Because capacity of the hull is a continuous function for the topology of uniform convergence, we see that, when parametrizing all curves by capacity, $\g_\e^\d$ converges uniformly towards $\g_\e$. 

It is now easy to conclude that the curve $\g_\e$ is driven by $W$ (see e.g. \cite{LSW_LERW}, Proposition 3.14).
\end{proof}

\begin{rem}
As a consequence of Theorem \ref{thm:cv-explbubble}, we see that the loop traced by $\SLE_8(2)$ is independent of the starting point $a$.
\end{rem}

\subsection{Convergence of the Radon-Nikodym derivatives}\label{cv-det}

Let us consider an annulus $A$, and let $K$ be a nice\footnote{We assume as before that $A\sm K$ is (simply-)connected, and that $a\in\partial(A\sm K)$.} compact subset of $A$ with a marked point on its boundary, so that its boundary is split in a Dirichlet subarc $\A_K$, and a Neumann subarc $\B_K$. Let $c$ be a cut disjoint from $K$, such that $A\sm c$ is simply-connected, and consider two cuts $d_1$ and $d_2$ that disconnect $c$ from the compact $K$ (see Figure \ref{fig-Setup1}). Moreover, we fix an interior point $x_0$ between $c$ and $d_1$. For any mesh size $\d$, we consider natural approximations of this set-up.

We are interested in the discrete harmonic operators, from the space of functions on $d^\d$ to itself, associated to the kernels $\T^\d = - \Ha^{c^\d\rightarrow d^\d}_{A^\d} \circ \Ha^{d^\d\rightarrow c^\d}_{A^\d}$ and $\T^\d_{K^\d} = - \Ha^{c^\d\rightarrow d^\d}_{A^\d\sm{K^\d}} \circ \Ha^{d^\d\rightarrow c^\d}_{A^\d\sm{K^\d}}$.

\begin{prop}\label{prop:cv-det}
The determinants of $\Id + \T^\d$ and $\Id+\T_{K^\d}^\d$ converge uniformly on compact sets of decorated domains towards the Fredholm determinants of their continuous counterparts.
\end{prop}

\begin{proof}
The determinant of $\Id+\T^\d$ can be expanded in the following way:
$$
\det(\Id+\T^\d) = \sum_{n\geq 1} \frac{1}{n!} \sum_{x_1,...,x_n} \det\left([\T^\d(x_i,x_j)\mu^{A^\d \sm d^\d}_{x^\d_0}(\{x_j\})]_{i,j}\right),
$$
where $\T^\d(x_i,x_j)$ is the kernel associated to the operator $\T^\d$ on $\mathcal{L}^2\left(c^\d,\mu^{A^\d \sm d^\d}_{x^\d_0}\right)$.

The function $\T^\d$ is locally uniformly bounded\footnote{i.e. uniformly bounded on Carath\'eodory-compact sets of decorated domains.} (a consequence of Lemma \ref{lem:poisson-kernel-bounded}), so the series expansion of the determinant is locally uniformly absolutely convergent\footnote{Consider a square matrix $\T$ of size $n$, which coefficients are bounded by $M$. Then Hadamard's inequality gives us the bound $|\T| \leq M^n n^{n/2}$.}.

We thus only need to prove that each term of the sum converges locally uniformly. Let us rewrite these terms:
\begin{eqnarray}
\sum_{x_1,\cdots,x_n} \det\left([\T^\d(x_i,x_j)\mu^{A^\d \sm d^\d}_{x^\d_0}(\{x_j\})]_{i,j}\right) &=& \sum_{x_1,\cdots,x_n} \sum_{\sigma \in S_n} \e(\sigma) \prod_{1\leq i \leq n}{\T^\d(x_i,x_{\sigma(i)})\mu^{A^\d \sm d^\d}_{x^\d_0}(\{x_{\sigma(i)}\})}. \nonumber
\end{eqnarray}
Hence, we just have to show, for any permutation $\sigma \in S_n$, the convergence of the following integral: 
$$
\int_{x_1,\cdots,x_n} \prod_{1\leq i \leq n}{\T^\d(x_i,x_{\sigma(i)})} \ \mu^{A^\d \sm d^\d}_{x^\d_0}(dx_1)\cdots \mu^{A^\d \sm d^\d}_{x^\d_0}(dx_n).
$$
This integral splits as a product of integrals over the cycles of the permutation $\sigma$, and we can thus reduce to proving, for any integer $l$, local uniform convergence of the following integral: 
$$
\int_{x_1,\cdots,x_l} \prod_{1\leq i \leq l}{\T^\d(x_i,x_{i+1})} \ \mu^{A^\d \sm d^\d}_{x^\d_0}(dx_1)\cdots \mu^{A^\d \sm d^\d}_{x^\d_0}(dx_l).
$$

We similarly reduce the convergence of the determinant $\det(\Id+\T^\d_{K^\d})$ to the local uniform convergence of the integrals
$$
\int_{x_1,\cdots,x_l} \prod_{1\leq i \leq l}{\T^\d_{K^\d}(x_i,x_{i+1})} \ \mu^{A^\d \sm (K^\d \cup d^\d)}_{x^\d_0}(dx_1)\cdots \mu^{A^\d \sm (K^\d \cup d^\d)}_{x^\d_0}(dx_l).
$$

These convergences follow from Lemma \ref{lem:cvintegrals}.
\end{proof}

\begin{lem}\label{lem:cvintegrals}
The integral quantities
$$
\int_{x_1,\cdots,x_l} \prod_{1\leq i \leq l}{\T^\d(x_i,x_{i+1})} \ \mu^{A^\d \sm d^\d}_{x^\d_0}(dx_1)\cdots \mu^{A^\d \sm d^\d}_{x^\d_0}(dx_l)
$$
and
$$
\int_{x_1,\cdots,x_l} \prod_{1\leq i \leq l}{\T^\d_{K^\d}(x_i,x_{i+1})} \ \mu^{A^\d \sm (K^\d \cup d^\d)}_{x^\d_0}(dx_1)\cdots \mu^{A^\d \sm (K^\d \cup d^\d)}_{x^\d_0}(dx_l)
$$
converge uniformly on Carath\'eodory-compact sets of decorated domains.
\end{lem}

\begin{proof}
We only discuss the case $l=1$ for the first integral, the general case being handled similarly.

Let us write out explicitly the integral we are interested in:
$$
\int_{y \in d^\d} \T^\d(y,y) \ \mu^{A^\d \sm d^\d}_{x^\d_0}(dy) = \int_{x\in c^\d, y \in d^\d} P^{A^\d\sm c^\d}_{x^\d_0} (x,y) P^{A^\d\sm d^\d}_{x^\d_0}(y,x) \ \mu^{A^\d \sm c^\d}_{x^\d_0}(dx)\mu^{A^\d \sm d^\d}_{x^\d_0}(dy).
$$

Fix some small enough $\e >0$. We split the cut $c$ in $n$ consecutive intervals $c_1,\cdots,c_n$, each coming with a marked point $x_i$. We choose this partition such that the intervals $c_i$ are of diameter less than $\e^3$ for $1<i<n$, such that $\mu^{A \sm d}_{x_0}(c_1\cup c_n)<\e$ and such that, for all $y\in d$, the continuous Poisson kernels $P^{A\sm d}_{x_0}(y,x)$ varies by less than $\e$ on each $c_i$. Let us also assume that the distance between $c_i$ and $\partial A$ is more than $\e^2$ for any $1<i<n$.

Note that on a Carath\'eodory-compact set of decorated domain $(A,x_0,c,d)$, we can always find such a decomposition in a number of intervals uniformly bounded by some integer $n(\e)$ independent of the domain. Moreover there is a uniform bound $M$ on the values taken by the continuous and discrete Poisson kernels $P^{A\sm c}_{x_0}(x,y)$ and $P^{A\sm d}_{x_0}(y,x)$ when $(x,y)$ runs over $c\times d$ (by Carath\'eodory continuity of the continuous Poisson kernel, respectively by Lemma \ref{lem:poisson-kernel-bounded} in the discrete case). There also is a uniform lower bound on the distance between the cuts $c$ and $d$, that we can assume to be bigger than $\e^2$.

We similarly split $d$ in consecutive intervals $d_1,\cdots,d_m$, and consider points $y_j$ in each of these intervals.

We can then approximate our integrals by a Riemann sum (see Lemma \ref{lem:approx-int-sum}):
$$
\left|\int_{x\in c^\d, y \in d^\d} P^{A^\d\sm c^\d}_{x^\d_0} (x,y) P^{A^\d\sm d^\d}_{x^\d_0}(y,x) \ \mu^{A^\d \sm c^\d}_{x^\d_0}(dx)\mu^{A^\d \sm d^\d}_{x^\d_0}(dy)-\sum_{i,j}\mu^{A^\d \sm c^\d}_{y^\d_j}(c^\d_i)\mu^{A^\d \sm d^\d}_{x^\d_i}(d^\d_j)\right|\leq 16 M^2\e + 2c M^2\e
$$
and
$$
\left|\int_{x\in c, y \in d} P^{A\sm c}_{x_0} (x,y) P^{A\sm d}_{x_0}(y,x) \ \mu^{A \sm c}_{x_0}(dx)\mu^{A \sm d}_{x_0}(dy)- \sum_{i,j}\mu^{A \sm c}_{y_j}(c_i)\mu^{A \sm d}_{x_i}(d_j)\right|\leq 2 M \e.
$$

For a fixed $\e$, the integers $n$ and $m$ are uniformly bounded on Carath\'eodory-compact sets of decorated domains. We can thus conclude by uniform convergence of harmonic measure of intervals (Lemma \ref{lem:harm-meas-conv} and Remark \ref{rem:harm-meas-conv-bound}).
\end{proof}

\begin{lem}\label{lem:approx-int-sum}
We have
$$
\left|\int_{x\in c^\d, y \in d^\d} P^{A^\d\sm c^\d}_{x^\d_0} (x,y) P^{A^\d\sm d^\d}_{x^\d_0}(y,x) \ \mu^{A^\d \sm c^\d}_{x^\d_0}(dx)\mu^{A^\d \sm d^\d}_{x^\d_0}(dy)-\sum_{i,j}\mu^{A^\d \sm c^\d}_{y^\d_j}(c^\d_i)\mu^{A^\d \sm d^\d}_{x^\d_i}(d^\d_j)\right|\leq 16 M^2\e + 2c M^2\e
$$
and
$$
\left|\int_{x\in c, y \in d} P^{A\sm c}_{x_0} (x,y) P^{A\sm d}_{x_0}(y,x) \ \mu^{A \sm c}_{x_0}(dx)\mu^{A \sm d}_{x_0}(dy)- \sum_{i,j}\mu^{A \sm c}_{y_j}(c_i)\mu^{A \sm d}_{x_i}(d_j)\right|\leq 2 M \e.
$$
\end{lem}

\begin{proof}
Let us first discuss the discrete statement. We will use the following relationship between Poisson kernel and harmonic measures:
$$
\int_{x\in c^\d_i, y \in d^\d_j} P^{A^\d\sm c^\d}_{x^\d_0} (x,y^\d_j) P^{A^\d\sm d^\d}_{x^\d_0}(y,x^\d_i) \ \mu^{A^\d \sm c^\d}_{x^\d_0}(dx)\mu^{A^\d \sm d^\d}_{x^\d_0}(dy) = \mu^{A^\d \sm c^\d}_{y^\d_j}(c^\d_i)\mu^{A^\d \sm d^\d}_{x^\d_i}(d^\d_j).
$$
Note that if $1<i<n$, by the Harnack inequality \ref{lem:harnack}, we have
\begin{eqnarray}
\left|\int_{x\in c^\d_i, y \in d^\d} P^{A^\d\sm c^\d}_{x^\d_0} (x,y) [P^{A^\d\sm d^\d}_{x^\d_0}(y,x)-P^{A^\d\sm d^\d}_{x^\d_0}(y,x^\d_i)] \ \mu^{A^\d \sm c^\d}_{x^\d_0}(dx)\mu^{A^\d \sm d^\d}_{x^\d_0}(dy)\right| \nonumber \\
\leq \int_{x\in c^\d_i, y \in d^\d} P^{A^\d\sm c^\d}_{x^\d_0} (x,y) [c \frac{\e^3}{\e^2}P^{A^\d\sm d^\d}_{x^\d_0}(y,x_i^\d)] \ \mu^{A^\d \sm c^\d}_{x^\d_0}(dx)\mu^{A^\d \sm d^\d}_{x^\d_0}(dy) \leq c M^2\e \mu^{A^\d \sm c^\d}_{x^\d_0}(c^\d_i). \nonumber
\end{eqnarray}

And we can sum these inequalities over $i$ (as $\sum_i \mu^{A^\d \sm c^\d}_{x^\d_0}(c^\d_i) \leq \mu^{A^\d \sm c^\d}_{x^\d_0}(c^\d) \leq 1 $).

Moreover, for the boundary terms, note that for $\d$ small enough:
$$
\left|\int_{x\in c^\d_1, y \in d^\d} P^{A^\d\sm c^\d}_{x^\d_0} (x,y) [P^{A^\d\sm d^\d}_{x^\d_0}(y,x)-P^{A^\d\sm d^\d}_{x^\d_0}(y,x^\d_1)] \ \mu^{A^\d \sm c^\d}_{x^\d_0}(dx)\mu^{A^\d \sm d^\d}_{x^\d_0}(dy)\right| \leq 2 M^2 \mu^{A \sm c}_{x^\d_0}(c^\d_1) \leq 4 M^2 \e,
$$
thanks to the convergence of the discrete harmonic measure $\mu^{A^\d \sm c^\d}_{x^\d_0}(c^\d_1)$ towards $\mu^{A \sm c}_{x_0}(c_1) \leq \e$ (Lemma \ref{lem:harm-meas-conv}).

Summing inequalities of this type yields the claim.

\hspace{.3cm}

The continuous statement is easier, since uniform continuity of the Poisson kernel gives, for $1\leq i\leq n$:
$$
\left|\int_{x\in c_i, y \in d} P^{A\sm c}_{x_0} (x,y) [P^{A\sm d}_{x_0}(y,x)-P^{A\sm d}_{x_0}(y,x_i)] \ \mu^{A \sm c}_{x_0}(dx)\mu^{A \sm d}_{x_0}(dy)\right| \leq M\e\mu^{A \sm c}_{x_0}(c_i).
$$
\end{proof}

\begin{lem}\label{lem:detisnonzero}
The determinants $\det_F(\Id + \T)$ and $\det_F(\Id + \T_{K})$ are positive and continuous as functions of decorated domains for the Carath\'eodory topology.
\end{lem}

\begin{proof}
The operator $\T$ being a strict contraction of the space of functions on $d$ equipped with the supremum norm, it is possible to build explicitly an inverse for the operator $\Id + \T$ - namely $(\Id + \T)^{-1} = \sum (-1)^n \T^n$ - and Fredholm determinants of invertible operators are non-zero (Theorem 3.5 b) in \cite{SimTrace}). Continuity of the determinants follows from the Carath\'eodory-continuity of harmonic measures and Poisson kernels.
\end{proof}

\subsection{Convergence of the exploration process $e^\d_{A^\d}$}\label{cv-exp-process}

We are now ready to prove the convergence of the exploration process $e^\d_{A^\d}$.

Recall that the processes $e^\d_{A^\d}$ form a tight family (Proposition \ref{prop:precompexpl}). It is hence sufficient to uniquely characterize the law of any subsequential limit $e_A$. We may assume (via Skorokhod's representation theorem) that the convergence is almost sure.
Let us consider the first time $T$ (resp. $T^\d$) when the trace of $e_A([0,T])$ (resp. $e^\d_{A^\d}([0,T^\d])$) disconnects the inner boundary of $A$ from the starting point $a$.
\begin{figure}[htb]
\begin{center}
\includegraphics[width = 8cm]{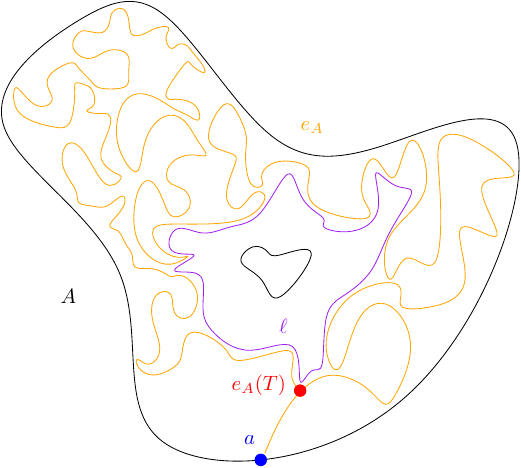}
\caption{The time $T$.}
\end{center}
\end{figure}

\begin{lem}\label{lem:cv-explafterT}
After time T, $e_A$ behaves as chordal $\SLE_8$ aimed at $a$ in the remaining domain.
\end{lem}

\begin{proof}
The time $T$ is the the limit of the times $T^\delta$, and moreover is the first time when the complement of $e_A([0,T])$ is disconnected (both facts follow from Proposition \ref{prop:nohidingloops}).
The complement of $e_A([0,T])$ then almost surely consists of two connected components. One of this connected components has the inner boundary of $A$ as part of its boundary. The exploration process $e_A$ is never going to visit this domain\footnote{This region would be covered by an exploration process started at a point on the inner boundary.}. The other connected component is a simply-connected open set with two marked boundary points $e_A(T)$ and $a$ on its boundary, with natural wired (resp. free) boundary conditions on the counterclockwise boundary arc $\arcperso{e_A(T),a}$ (resp. $\arcperso{a,e_A(T)}$).

As in the proof of Theorem \ref{thm:cv-explbubble}, we can use Proposition \ref{prop:USTcvSLE} to show that after time $T$, $e_A$ has the law of SLE$_8$ aimed at $a$ in this remaining domain. 
\end{proof}

We thus only have to characterize the behavior of $e_A$ up to time $T$. In order to do this, let us consider the hitting time $T_{c}^{\e}$ of the $\e$-neighborhood of a cut $c$.

\begin{lem}\label{lem:charact-cv-expl}
The law of $e_A$ stopped at all of the times $T_{c}^{\e}$ is enough to characterize the law of $e_A$ until the disconnection time $T$ of the point $a$ from the inner boundary of the annulus.
\end{lem}

\begin{proof}
Let us fix a countable family of cuts $\C$, that is dense in the set of $\C^1$ cuts equipped with the topology of uniform convergence up to reparametrization, and consider an enumeration $n\mapsto(c_n,\e_n)$ of the set of couples $\C \times 2^{-\N}$. We call $\widetilde{T}_n=T_{c_n}^{\e_n}$ the stopping time corresponding to the $n$-th couple.

Let us consider the family of stopping times $T_N=\sup_{n\leq N}\widetilde{T}_{n} \nearrow \sup_n \widetilde{T}_n = T$. Indeed, it is equivalent to have a curve $\gamma$ disconnect the inner boundary and the outer boundary (time $T$) or to have $\gamma$ touch any curve that connects the two boundaries (which is time $\sup_n \widetilde{T}_n$).

If the law of $e_A$ until the times $T_N$ and $\widetilde{T}_{N+1}$ is known, we can deduce its law until the time $T_{N+1}$. Indeed, the event  $E = \{\widetilde{T}_{N+1} > T_N \}$ is $\F_{T_N}\cap\F_{\widetilde{T}_{N+1}}$-measurable, and the law of $e_A$ until time $T_{N+1}$ is given on $E$ by the law of $e_A$ until time $\widetilde{T}_{N+1}$ conditioned on $E$, and on $E^c$ by the law of $e_A$ until time $T_N$ conditioned on $E^c$.
\end{proof}

We are now ready to prove Theorem \ref{thm:cv-expl}.

\begin{thm}\label{thm:cv-expl}
The exploration process $e^\d_{A^\d}$ converges in law (for the topology $\Top$), as $\d$ goes to $0$, towards a continuous process $e_A$ characterized by the following property. Until the first hitting time $T_{c}^{\e}$ of the $\e$-neighborhood of a cut $c$, $e_A$ is absolutely continuous with respect to counterclockwise $\SLE_8(2)$ started at $a$ in the domain $A\sm c$, with a Radon-Nikodym derivative given by:
$$
\frac{{\det}_F(\Id - \Ha^{c\rightarrow d}_{A\sm{K}} \circ \Ha^{d\rightarrow c}_{A\sm{K}})} {{\det}_F(\Id - \Ha^{c\rightarrow d}_{A} \circ \Ha^{d\rightarrow c}_{A})},
$$
where $K$ is the trace of the exploration process $e_A([0,T_{c}^{\e}])$, and boundary conditions are Neumann on the counterclockwise boundary arc $\arcperso{a,e_A(T_{c}^{\e})}$ and Dirichlet elsewhere.
\end{thm}

\begin{proof}
Before time $T_{c^\d}^{\e}$, the discrete exploration process $e^\d_{A^\d}$ in the annulus is absolutely continuous with respect to the process $e^\d_{A^\d\sm c^\d}$, and the Radon-Nikodym derivative is known (Lemma \ref{lem:DRN-expl}). We proved convergence of the reference exploration process $e^\d_{A^\d\sm c^\d}$ (Theorem \ref{thm:cv-explbubble}) and uniform convergence on Carath\'eodory-compact sets of the Radon-Nikodym derivative (Proposition \ref{prop:cv-det} and Lemma \ref{lem:detisnonzero}). Hence the law of any subsequential limit $e_A$ until time $T_{c}^{\e}$ is uniquely characterized. Lemmas \ref{lem:charact-cv-expl} and \ref{lem:cv-explafterT} allow us to conclude.
\end{proof}

\section{An $\SLE_2$ loop measure}\label{loop}

\subsection{On the loop measure $\mu^\#_A$}\label{loop-prop}

We may now discuss the loop measure we are interested in. Let $\l^\d_{A^\d}$ be the boundary of the discrete exploration process $e^\d_{A^\d}$.

\begin{thm1}
The random loop $\l^\d_{A^\d}$ converges in law (for the topology $\Top$ of uniform convergence up to reparametrization) towards a simple curve $\l_A$ whose range is almost surely the boundary of the continuous exploration process $e_A$, which is also the interface between the inner and outer component of the wired UST in $A^\d$.
\end{thm1}

\begin{proof}
The family of curves $\l^\d_{A^\d}$ is tight (as in the proof of Proposition \ref{prop:precompexpl}). Any subsequential limit $\l_A$ has to contain the boundary of $e_A$ in its range. By Proposition \ref{prop:nohidingloops}, the loop $\l_A$ is almost surely simple. This gives the reverse inclusion: the range of $\l_A$ is the boundary of $e_A$.
\end{proof}

\begin{rem}\label{rem:outerisinner}
Working with the inner tree and its exploration process, one can recover $\l_A$ as the boundary of the inner exploration process.
\end{rem}

\begin{prop}\label{prop:loops-avoid-boundary}
The law $\mu^\#_A$ of the random loop $\l_A$ is supported on loops that do not touch the boundary $\partial A$. 
\end{prop}

\begin{proof}
This follow directly from the arm estimate Lemma \ref{lem:dualtrunkdisjointfromboundary}.
\end{proof}

\begin{rem}\label{rem:loopSLE2}
The measure $\mu^\#_{A}$ is locally (absolutely continuous with respect to) the boundary of an $\SLE_8$ process, and hence a version of $\SLE_2$ (\cite{Dub_dual}).
\end{rem}

Let us now investigate the conformal covariance of the family of measures $\mu^\#_A$, i.e. how these measures behave under conformal mappings.

\begin{prop}\label{prop:covar-conf}
Let $\phi$ be an injective holomorphic map from an annulus $A'$ to another annulus $A$ (where $\phi(A')$ is a retract of $A$).

We have the following absolute continuity formula:
$$
\frac{\text{\normalfont d} \phi_* \mu^\#_{A'} }{\text{\normalfont d}{\mu^\#_{A}}}(\l) = e^{-\M(A,\phi(A');\l)} \ind_{\l \subset \phi(A')}.
$$
\end{prop}

\begin{proof}
We prove this statement below when $\phi$ is an isomorphism (Lemma \ref{lem:covar-conf-isom}) and for subdomains $A'\subset A$ (Lemma \ref{lem:covar-conf-id}) sharing their outer boundaries with $A$. The general case follows from the cocycle property of $\M$.
\end{proof}

\begin{lem}\label{lem:covar-conf-isom}
The collection of loop measures $\mu^\#_{\cdot}$ has the following conformal invariance: given any conformal isomorphism $\phi : A \buildrel\sim\over\longrightarrow A'$, we have that $\phi_* \mu^\#_{A} = \mu^\#_{A'}$.
\end{lem}

\begin{proof}
The law of the exploration process $e_A$ is characterized by its law stopped at the times $T^\e_c$. In turn, these are absolutely continuous with respect to the conformally invariant process $\SLE_8(2)$, with a Radon-Nikodym derivative that can be expressed using integrals of harmonic quantities, so that is also conformally invariant. Hence, the law of the outer exploration process is invariant by any conformal map that sends outer boundary to outer boundary.

Moreover, the outer exploration process is sent to the inner exploration process by an inversion. Remark \ref{rem:outerisinner} allows us to conclude that the measures $\mu^\#_.$ are preserved by such maps.
\end{proof}

In particular, the family of measures $\mu^\#_{A}$ can be naturally extended to any annular subdomain of the plane, without any regularity assumption.

Let us now state how the family $\mu^\#_{A}$ behaves under restriction to a subdomain.

\begin{lem}\label{lem:covar-conf-id}
Let $A'=A \sm H$ where $H$ is a compact subset of $\overline{A}$ intersecting its inner boundary, so that $A'$ is an annulus that shares its outer boundary with $A$. Then:
$$
\frac{\text{\normalfont d} \mu^\#_{A'}}{\text{\normalfont d} \mu^\#_{A}} (\l)  = e^{-\M(A,A';\l)}.
$$
\end{lem}

\begin{figure}[htb]
\begin{center}
\includegraphics[width = 8cm]{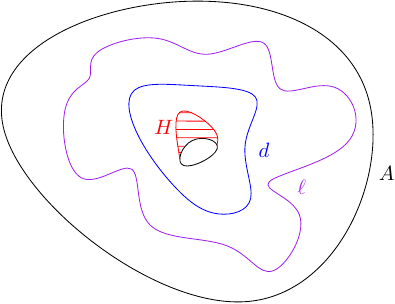}
\caption{The set-up of Lemma \ref{lem:covar-conf-id}.}
\end{center}
\end{figure}

\begin{proof}
Let us first assume that $A'$ is a smooth subdomain of $A$. The proof then follows those of Lemma \ref{lem:DRN-expl} and Proposition \ref{prop:cv-det}.

The corresponding Radon-Nikodym derivative in the discrete setting can be explicitly computed as a ratio of the total number of spanning trees on different graphs, which can be rewritten as a ratio determinants of Laplacians (where the curve $\l$ carries Neumann boundary conditions):
$$
\frac{ \text{\normalfont d} \mu^{\#\d}_{{A'}^\d}}{ \text{\normalfont d} \mu^{\#\d}_{A^\d}}(\l^\d) = \frac{\det(\D_{A^\d})\det(\D_{{A'}^\d\sm\l^\d})}{\det(\D_{{A'}^\d})\det(\D_{A^\d\sm\l^\d})}.
$$

Under $\mu^\#_{A'}$, the loop $\l$ almost surely does not touch the boundary (Proposition \ref{prop:loops-avoid-boundary}), so we can choose a smooth simple curve $d$ that disconnects $H$ from $\l$ in the annulus $A$. We can then rewrite (as in Section \ref{subsec:spandet}):
$$
\frac{ \text{\normalfont d} \mu^{\#\d}_{{A'}^\d}}{ \text{\normalfont d} \mu^{\#\d}_{A^\d}}(\l^\d) = \frac{\det(\Id - \Ha^{\partial H^\d\rightarrow d^\d}_{A^\d} \circ \Ha^{d^\d\rightarrow \partial H^\d}_{A^\d})} {\det(\Id - \Ha^{\partial H^\d\rightarrow d^\d}_{A^\d\sm\l^\d} \circ \Ha^{d^\d\rightarrow \partial H^\d}_{A^\d\sm\l^\d})}.
$$
The convergence of these determinants then follows from the work in Section \ref{cv-det} (Lemma \ref{lem:poisson-kernel-bounded} and Remark \ref{rem:harm-meas-conv-2-conn} allow us to work in the doubly-connected domains that appear here).

Finally, using Lemma \ref{lem:ratio-det} and Remark \ref{rem:ratio-det-non-smooth}, we get the claim for a smooth subdomain $A'$:
$$
\frac{\text{\normalfont d} \mu^\#_{A'}}{\text{\normalfont d} \mu^\#_{A}} (\l) = \frac{{\det}_F(\Id - \Ha^{\partial H\rightarrow d}_{A} \circ \Ha^{d\rightarrow \partial H}_{A})} {{\det}_F(\Id - \Ha^{\partial H\rightarrow d}_{A\sm\l} \circ \Ha^{d\rightarrow \partial H}_{A\sm\l})}= e^{-\M(A,A\sm H;\l)}.
$$

Now, a general subdomain $A'$ can always be approximated (in Carath\'eodory topology) by a sequence of smooth subdomains $A'_n$ of same modulus. Combining the invariance of the loop measure under conformal isomorphisms (Lemma \ref{lem:covar-conf-isom}) with the fact that the cocycle $\M(A,A';\l)$ is continuous in $\l$ and $A'$ (Remark \ref{rem:cont-cocycle}) allows us to conclude.
\end{proof}

\subsection{From annuli to Riemann surfaces}\label{from-ann-to-rs}

Let us now explain how to get, from the family of measures $\mu^\#_{A}$, a family of measures $\mu_\Si$ on loops on any Riemann surface satisfying a nice conformal covariance property. We will then recall how this covariance can be absorbed by an algebraic structure (namely the determinant line bundle): we will describe how to build from $\mu_\Si$ a (formally) conformally invariant family of measures taking values in (some power of) the determinant line bundle.

\subsubsection{Building an $\SLE_2$ loop measure}\label{subsec-BuildingSLE2}

\begin{prop}\label{prop:extloopmeas}
For any Riemann surface $\Si$, we can define a measure $\mu_\Si$ on the space of its simple loops ${\mc L}(\Si)$ by:
$$
\ind_{\l \in {\mc L}^\times(A)} \text{\normalfont d}{\mu_\Si} (\l)   = e^{\M(\Si, A;\l)} \text{\normalfont d}\mu^\#_{A}(\l),
$$
for any loop $\l$ that is topologically non-trivial in the conformal annulus $A \subset \Si$.
\end{prop}

\begin{proof}
The conformal covariance of the family $\mu^\#_{A}$ (Proposition \ref{prop:covar-conf}), as well as the cocycle property \ref{prop:cocycle} of the quantity $\M$ ensure this definition is consistent.
\end{proof}

The family thus defined satisfies, for any conformal embedding $\Si' \hookrightarrow \Si$, the following restriction covariance property:
$$
\frac{ \text{\normalfont d} \mu_{\Si'} }{ \text{\normalfont d}{\mu_\Si} }(\l) = e^{-\M(\Si,\Si';\l)}\ind_{\l \in {\mc L}(\Si')}.
$$

\subsubsection{Kontsevich-Suhov loop measures}

Let us now recall the set-up of Kontsevich and Suhov (\cite{KontSuh}).

\begin{defn}\label{def:c-lcc}
A $c$-locally conformally covariant ($c$-lcc) loop measure is a collection $(\lambda_\Sigma)_\Sigma$ indexed by all Riemann surfaces $\Sigma$, such that $\lambda_\Sigma$ is a $|\Det|_\Sigma^{\otimes c}$-valued measure on ${\mc L}(\Sigma)$ and such that for any embedding $\xi:\Sigma'\hookrightarrow\Sigma$,
$$
\xi^*\lambda_{\Sigma}=\lambda_{\Sigma'}.
$$
\end{defn}

For our purposes, we canonically write the fiber of the determinant line bundle
$$
|\det|_{\Sigma,\ell}\simeq|\det|_\Si\otimes|\det|_{\Si\setminus\ell}^{-1},
$$
and then pick the element $s_\Si(\ell)=(\psi_{\Si}\otimes\psi_{\Si\sm\l}^{-1})^{\otimes - \frac{1}{2}}$ in the RHS, where the element $\psi$ of $|\det|^{-2}$ is the area-corrected Laplacian $\zeta$-determinant with boundary conditions Dirichlet on $\partial\Sigma$ and Neumann on $\ell$ (see (\ref{eq:psi1}) and (\ref{eq:psi2})). In other words, we trivialize the line bundle $|\Det|_\Sigma^{\otimes c}$ using the functional given on normalized (and well-behaved) metrics by $s^c_\Sigma(\ell)=(\det_\z(\D_{\Si\setminus\ell})/\det_\z(\D_\Si))^{c/2}$. Indeed, recall that $\zeta$-determinants are positive by definition, and so $s^c_\Sigma$ is a well-defined non-vanishing section of the line bundle $|\Det|_\Sigma^{\otimes c}$.

If we set $\mu_\Si=s_\Si^{-c}\lambda_\Si$, we obtain a (scalar) measure on ${\mc L}(\Si)$. The $c$-lcc property of the collection $\lambda_.$ then becomes (for an embedding $\Si' \hookrightarrow \Si$),

$$
\frac{d\mu_{\Si'}}{d\mu_\Si}(\ell)=\frac{s^c_{\Si}}{s^c_{\Si'}}(\ell)\ind_{\l \in {\mc L}(\Si')}=\left(\frac{\det_\z(\D_{\Si'})\det_\z(\D_{\Si\setminus\ell})}{\det_\z(\D_{\Si})\det_\z(\D_{\Si'\setminus\ell})}\right)^{c/2}\ind_{\l \in {\mc L}(\Si')}=e^{\frac{c}{2}\M(\Si,\Si';\l)}\ind_{\l \in {\mc L}(\Si')},
$$
where the ratio of determinants is evaluated on a neutral collection of normalized metrics.

Moreover, remark that the space of simple loops ${\mathcal L}(\Sigma)$ can be covered by the sets ${\mc L}^\times(A)$ (that consists of simple loops that generate $\pi_1(A)$), where $A$ is an embedded annulus in $\Sigma$. Hence, by general constructive measure theory, a $c$-lcc loop measure is completely characterized by the data, for any annulus $A$ of the restriction $\lambda^\times_A$ of $\lambda_A$ to the set of loops ${\mc L}^\times(A)$. Conversely, a $c$-lcc loop measure can be constructed given any collection of measures $\lambda^\times_A$ that satisfies the restriction for inclusion of annuli $A'\hookrightarrow A$ (where $A'$ is a retract of $A$). 

This is nothing more but a more abstract phrasing of the procedure we followed in Section \ref{subsec-BuildingSLE2} to build, from a family of measures $\mu^\#_A$ satisfying a conformal covariance property, a family of loop measures on all Riemann surfaces. In this abstract language, the $\SLE_2$ loop measure we built corresponds to a $c$-lcc family of measures $\lambda_\Si$ with parameter $c=-2$. 

\begin{thm0}
There exists a $-2$-lcc loop measure.
\end{thm0}

\begin{proof}
By the procedure described above, the existence of a $-2$-lcc loop measure $\lambda_\Si$ is a direct consequence of Proposition \ref{prop:covar-conf}. Note that we directly constructed the family of scalar measures $\mu_\Si$ corresponding to $\lambda_\Si$ via the trivializations $s_\Si$ in Proposition \ref{prop:extloopmeas}.
\end{proof}

The family of scalar measures $\mu_\Si$ that can be obtained from a given $c$-lcc loop measure $\lambda_\Si$ is in no way unique. The trivializations $s_\Si$ we used have a nice property: for any annulus $A$, the measure $\ind_{\l \in {\mc L}^\times(A)} \text{\normalfont d}\mu_A(\l)$ is a probability measure. However, such a requirement is far from characterizing $\mu_\Si$ uniquely (or even the measures $\mu^\#_A$ for that matter). Other choices of trivializations of the determinant line bundle would yield a family of scalar measures satisfying a restriction covariance property with a cocycle $\widetilde{\M}$ differing from $\M$ by a coboundary.

In particular, we could have built other natural families of scalar measures $\mu_\Si$ with a restriction property given by a cocycle $\widetilde{\M}$ corresponding to some other regularization of masses of Brownian loops (recall Remark \ref{rem:MregBl}). For example, Wendelin Werner's $\SLE_{8/3}$ loop measure (\cite{Wer_loops}) provides a probabilistic regularization of the Brownian loop measure. Another regularization was introduced by Field and Lawler in \cite{FieLawBrownianloop}; their method would add a tensorial dependency at a marked interior point.

\paragraph{Acknowledgements}
We would like to thank the referee for her/his careful reading as well as for her/his helpful suggestions for improvements.

\bibliography{biblio}{}
\bibliographystyle{plain}

----------------------------------

\noindent Columbia University\\
Department of Mathematics\\
2990 Broadway\\
New York, NY 10027

\end{document}